\documentclass[10pt]{amsart}
\usepackage[utf8]{inputenc}
\usepackage{bm}
\usepackage{fontenc}
\usepackage{amsfonts}
\usepackage{amssymb}
\usepackage{amsmath}
\usepackage{amsthm}\usepackage{mathtools}
\usepackage{enumerate}
\usepackage{hyperref}\usepackage{enumitem}
\usepackage{mathrsfs}
\usepackage{tikz}\usetikzlibrary{calc}
\usepackage{marginnote}
\usepackage{xcolor,enumitem}
\usepackage{soul}\usepackage{tikz}
\usepackage[all,cmtip]{xy}
\usepackage{tikz,tikz-cd}

\newcommand{\N}{\mathbb{N}}

\newcommand{\cY}{\mathcal{Y}}
\newcommand{\cX}{\mathcal{X}}

\newcommand{\NN}{\mathbb{N}}

\newcommand{\cstu}{\mathrm{C}^*_u}

\newtheorem*{rigprob*}{Rigidity Problem for uniform Roe Algebras}
\newtheorem*{rigprobcorona*}{Rigidity Problem for uniform Roe Coronas}

\newcommand{\cWEP}{\mathsf{cwEP}}

\newcommand{\cstar}{$\mathrm{C}^*$}

\newcommand{\bbN}{\mathbb{N}}

\DeclareMathOperator{\OCA}{{\mathsf {OCA}}}
\DeclareMathOperator{\PFA}{{\mathsf {PFA}}}
\DeclareMathOperator{\WEP}{{\mathsf {wEP}}}

\DeclareMathOperator{\MA}{{\mathsf {MA}_{\aleph_1}}}

\DeclareMathOperator{\ZFC}{{\mathsf {ZFC}}}

\newcommand{\CH}{\mathsf{CH}}

\newtheorem{theorem}{Theorem} [section]
\newtheorem*{theorem*}{Theorem}
\newtheorem{prop}[theorem]{Proposition}

\newtheorem*{proposition*}{Proposition}
\newtheorem{lemma}[theorem]{Lemma}
\newtheorem*{lemma*}{Lemma}
\newtheorem{corollary}[theorem]{Corollary}
\newtheorem*{corollary*}{Corollar}

\newtheorem*{fact*}{Fact}
\theoremstyle{definition}
\newtheorem{definition}[theorem]{Definition}
\newtheorem*{definition*}{Definition}
\newtheorem{claim}[theorem]{Claim}
\newtheorem*{claim*}{Claim}

\newtheorem*{conjecture*}{Conjecture}

\newtheorem{question}[theorem]{Question}

\theoremstyle{remark}
\newtheorem{example}[theorem]{Example}
\newtheorem*{example*}{Example}
\newtheorem{remark}[theorem]{Remark}
\newtheorem*{remark*}{Remark}

\newtheorem*{note*}{Note}
\newtheorem*{question*}{Question}

\newcommand{\norm}[1]{\left\lVert #1 \right\rVert}

\DeclareMathOperator{\supp}{supp}

\DeclareMathOperator{\Fin}{Fin}

\newcounter{my_enumerate_counter}
\newcommand{\pushcounter}{\setcounter{my_enumerate_counter}{\value{enumi}}}
\newcommand{\popcounter}{\setcounter{enumi}{\value{my_enumerate_counter}}}

\usepackage{enumitem}

\title{Rigidity of Higson coronas}

\author{Alessandro Vignati}
\address[AV]{
 Institut de Math\'ematiques de Jussieu - Paris Rive Gauche (IMJ-PRG)\\
 Universit\'e Paris Cit\'e\\ Institut Universitaire de France\\
 B\^atiment Sophie Germain\\
 8 Place Aur\'elie Nemours \\ 75013 Paris, France}
\email{ale.vignati@gmail.com}
\urladdr{http://www.automorph.net/avignati}

\date{\today}

\begin{document}
\maketitle

\begin{abstract}
We show that under mild set theoretic hypotheses we have rigidity for algebras of continuous functions over Higson coronas, topological spaces arising in coarse geometry. In particular, we show that under $\mathsf{OCA}$ and $\mathsf {MA}_{\aleph_1}$, if two uniformly locally finite metric spaces $X$ and $Y$ have homeomorphic Higson coronas $\nu X$ and $\nu Y$, then $X$ and $Y$ are coarsely equivalent, a statement which provably does not follow from $\mathsf{ZFC}$ alone.
\end{abstract}

\section{Introduction}

The Higson compactification and corona were introduced by Higson in the late '80s in connection with a $K$-theoretic analysis of index theorems for noncompact Riemannian manifolds, even though they appeared in published form slightly later, in \cite{Higson.ExtDuality} and \cite[\S5]{Roe.CoarseCoh}, where they have been generalized to locally compact metric spaces. These topological spaces have in coarse geometry a similar role of that of the \v{C}ech--Stone compactification and remainder in topology, as they code certain asymptotic behaviour of the metric spaces of interest.

Higson coronas found profound applications within and outside coarse geometry: for example, their cohomology theory is key in the study of hypereuclidean manifolds (\cite{Roe.CoarseCoh} or \cite{Drani.HigsonRoe}) and they have been used in the study of asymptotic dimension (\cite{Drani.AsymDim}) and in applications to the Baum--Connes conjecture (\cite{FukayaOguni}). Their stable versions, introduced by Emerson and Meyer in \cite{EmersonMeyer.Dual}, were used for a Dirac-dual-Dirac approach to the Baum--Connes conjecture (\cite{EmersonMeyer.Dirac}), and are heavily present in the work of Willett for their homological properties, in relation with the Novikov conjecture (\cite{Willett.Thesis,Willett.Homological}).

Even though the definitions of Higson compactification and corona were given in the general locally compact setting, here we focus on discrete bounded geometry metric spaces, meaning that all spaces of interest are \emph{uniformly locally finite} (u.l.f.\ from now on) i.e., bounded balls of a given radius are uniformly bounded in size. Typical examples of u.l.f.\ spaces that are important for applications are finitely generated groups with word metrics, and discretizations of non-discrete spaces such as Riemannian manifolds.

\begin{definition}
Let $(X,d)$ be a u.l.f.\ metric space. A bounded function $f\colon X\to\mathbb C$ is \emph{slowly oscillating}\footnote{These are also called Higson functions, or functions of bounded variation.} if for every $\varepsilon>0$ and $R>0$ there is a finite set $F\subseteq X$ such that $|f(x)-f(x')|<\varepsilon$ whenever $x$ and $x'$ are points in $X\setminus F$ with $d(x,x')\leq R$. 

The algebra of slowly oscillating functions, denoted $C_h(X)$, is a unital \cstar-subalgebra of $\ell_\infty(X)$, and we call its spectrum $hX$ the \emph{Higson compactification} of $X$. The \emph{Higson corona}, $\nu X=hX\setminus X$, is the spectrum of the quotient algebra $C_\nu(X)=C_h(X)/C_0(X)$.
\end{definition}
The association from metric spaces to Higson coronas is functorial, in the setting of coarse categories. This is formalised by the following, which is \cite[Proposition 2.41]{RoeBook}. (For the appropriate terminology for maps in coarse geometry, see Definition~\ref{def:maps}.)
\begin{prop}\label{prop:roetrivial}
Let $X$ and $Y$ be u.l.f. spaces. A coarse proper map $\varphi\colon Y\to X$ gives a continuous map $\nu\varphi\colon\nu Y\to\nu X$. If $\varphi,\psi\colon Y\to X$ are close, then $\nu\varphi=\nu\psi$. Moreover, if $\varphi$ is a coarse embedding, $\nu\varphi$ is injective and if $\varphi$ is a coarse equivalence, $\nu\varphi$ is a homeomorphism.
\end{prop}

In these notes, we focus on the converse problem.
\begin{question}\label{Q.main}
How much of the coarse geometry of u.l.f.\ metric spaces is remembered by Higson coronas? Specifically, for $X$ and $Y$ two u.l.f. metric spaces, if $C_{\nu}(X)$ is isomorphic to $C_{\nu}(Y)$, are $X$ and $Y$ coarsely equivalent?
\end{question}

While this question was studied for general locally compact metric spaces when one considers the $C_0$-coarse structure, all u.l.f.\ metric spaces have coarsely equivalent $C_0$ coarse structures, therefore the results of \cite{MineYamashita} do not apply here. In fact, no (even partial) positive answer to the second part of Question~\ref{Q.main} has been given so far. 

Question~\ref{Q.main} follows the pattern of \emph{rigidity problems for Roe-like algebras}. Named after John Roe who introduced prototypical versions of them in \cite{Roe:1988qy} for index theoretic purposes, Roe-like algebras are \cstar-algebras associated to metric spaces capable of catching algebraically some of the coarse geometric properties of metric spaces. These \cstar-algebras found applications in many different areas within and at the boundary of mathematics; we refer the reader to the introduction of \cite{SquareInventiones} for a thorough introduction on these objects and their applications. 

Example of Roe-like algebras are the uniform Roe algebra of $X$, $\cstu(X)$, the Roe algebra of $X$, $\mathrm{C}^*(X)$, the uniform Roe corona of $X$, denoted by $Q_u(X)$ and defined as the quotient of the uniform Roe algebra by the ideal of compact operators, and of course the algebra of slowly oscillating functions $C_h(X)$ and its quotient $C_\nu(X)$. The rigidity problem for Roe-like algebras essentially asks how much geometry is remembered by these \cstar-algebras. In recent year there have been important developments on these problems. Notably:
\begin{itemize}
\item It is easy to check that if $C_h(X)$ and $C_h(Y)$ are isomorphic, $X$ and $Y$ are bijectively coarsely equivalent\footnote{This is true even in the non-discrete setting, see \cite[\S4]{AlvarezCandel}.}. The reason for this is that an isomorphism between algebras of slowly oscillating functions necessarily maps minimal projections to minimal projections.
\item Following up work on Property A spaces of \v{S}pakula and Willett (\cite{SpakulaWillett2013}), it was shown in \cite{SquareInventiones} that if two u.l.f.\ metric space have isomorphic uniform Roe algebras, then the spaces are coarsely equivalent. (Indeed, it is conjectured that isomorphism of uniform Roe algebras gives bijective coarse equivalence between the underlying spaces. This conjecture has been verified for a large class of metric space, see the introduction of \cite{SquareBij}).
\item In the recent \cite{MartinezVigolo}, Martinez and Vigolo showed that if two Roe algebras are isomorphic, their underlying spaces are coarsely equivalent.
\end{itemize}

When dealing with quotient structures such as uniform Roe or Higson coronas set theory enters play. This should not be surprising, as often the behaviour of homomorphisms between massive quotient structure depends on the set theoretic ambient. This area of mathematics was developed following Shelah's proof (\cite{Sh:PIF}) that consistently with $\ZFC$ all automorphisms of the Boolean algebra $\mathcal P(\bbN)/\Fin$ are trivial, while the same statement is false if the Continuum Hypothesis $\CH$ is assumed (\cite{Ru}). Shelah's argument was generalised (\cite{ShSte:PFA,Ve:OCA,TrivIso}) to show that triviality of automorphisms of $\mathcal P(\bbN)/\Fin$ is a consequence of mild set theoretic hypotheses such as Todor\v{c}evi\`c's Open Colouring Axiom $\OCA$ (a combinatorial axiom \`a la Ramsey) and Martin's Axiom at level $\aleph_1$, $\MA$ (a generalisation of Baire category theorem to the uncountable). (For a brief introduction to these axioms see \S\ref{ss.ST}). These proofs were adapted to a variety of quotient structures (rings, Boolean algebras, ...) in the discrete setting, and the results of Farah (\cite{Fa:All}) and Phillips and Weaver (\cite{Phillips-Weaver}) on automorphisms of the Calkin algebra brought these ideas to the setting of \cstar-algebras and their quotients. We refer the reader to the long survey \cite{CoronaRigidity} for a nice overview of this exciting area of mathematics.

The analysis of massive quotient structures from the perspective of set theory was applied to Roe-like algebras. In \cite{BragaFarahV.Roecoronas} (see also \cite[\S3.4]{SquareInventiones}) it was shown that the uniform Roe corona remembers the geometry of the spaces of interest, once again if the axioms $\OCA$ and $\MA$ are assumed. On another direction, recent work of Brian and Farah (\cite{BrianFarah}) gives that $\CH$ implies the existence of $2^{\aleph_0}$ mutually non-coarsely equivalent u.l.f.\ metric spaces with isomorphic uniform Roe coronas, meaning in this setting the geometry is not remembered by the algebra of interest. These spaces have asymptotic dimension one\footnote{Here we refer to the asymptotic dimension in the sense of Gromov; this is the appropriate notion of dimension in the coarse setting, see \cite{Gromov93}. Relevant for us is that the asymptotic dimension of a metric space is remembered by the Higson corona, see \cite{Drani.AsymDim}}. When focusing on $C_\nu(X)$ (and therefore dually on Question~\ref{Q.main}), it has been noted in \cite{Protasov.Corona} that, again under $\CH$, all asymptotic dimension zero spaces have isomorphic Higson coronas\footnote{In retrospect, this is a fairly easy consequence of Parovi\v{c}enko's Theorem, see \cite{Pa:Universal}, see Example~\ref{example:nontrivial} for more details.}. Furthermore, since the \cstar-algebra $C_\nu(X)$ can be identified as the center of $Q_u^*(X)$ (\cite[Proposition 3.6]{SquareEmb}), the result of Brian and Farah gives, under $\CH$, a continuum of mutually non coarsely equivalent u.l.f.\ metric spaces of asymptotic dimension one whose Higson coronas are homeomorphic. Therefore, under $\CH$, Question~\ref{Q.main} has a negative answer.

The main result of this article proves that Question~\ref{Q.main} can have a positive answer.

\begin{theorem}\label{thm:main}
Assume $\OCA$ and $\MA$. Let $X$ and $Y$ be u.l.f.\ metric spaces. If $X$ and $Y$ have homeomorphic Higson coronas, then $X$ and $Y$ are coarsely equivalent.
\end{theorem}

To analyse continuous maps between two Higson coronas $\nu Y\to \nu X$, and ultimately prove Theorem~\ref{thm:main}, we study the structure of unital $^*$-homomorphisms $C_\nu(X)\to C_\nu(Y)$. We introduce a notion of triviality for these $^*$-homomorphisms (Definition~\ref{def:trivial}) which isolates those $^*$-homomorphisms arising from a coarse proper map as in Proposition~\ref{prop:roetrivial}. The idea is then to prove, under suitable axioms, that all isomorphisms between algebras of the form $C_\nu(X)$ are trivial. In fact, we do better: we isolate a principle which completely describes the structure of unital $^*$-homomorphisms between algebras of the form $C_\nu(X)$. This principle, which we call the \emph{coarse weak Extension Principle}, denoted $\cWEP$ and defined in Definition~\ref{def:cwep}, is the coarse version of the weak Extension Principle formulated by Farah (\cite{FarahBook2000}) for \v{C}ech--Stone remainders of zero-dimensional locally compact topological spaces and extended to locally compact spaces of arbitrary topological dimension in \cite{VY:wep}. 

The main technical content of this article goes in proving that this coarse weak Extension Principles holds under $\OCA$ and $\MA$. The proof uses powerful lifting results under $\OCA$ and $\MA$ obtained in \cite{mckenney2018forcing} and \cite{TrivIsoMetric}. As a byproduct, we get the following result characterising embeddings between Higson coronas. Recall that, for a u.l.f.\ metric space $Y$, if $U\subseteq\nu Y$ is clopen then there is a subspace $\tilde U\subseteq Y$ such that the the projection $\chi_{\tilde U}\in\ell_\infty(Y)$ is slowly oscillating and the image of $\chi_{\tilde U}$ in $\ell_\infty(Y)/C_0(Y)$ corresponds to the characteristic function on $U$ in $C_\nu(Y)$.

\begin{theorem}\label{thm:main2}
Assume $\OCA$ and $\MA$. Suppose that $\tilde\varphi\colon \nu Y\to \nu X$ is a continuous injection. Then there is $\tilde U\subseteq Y$ which gives a clopen $U\subseteq\nu Y$ such that there is a coarse embedding $\varphi\colon \tilde U\to X$ with $\nu\varphi=\tilde\varphi\restriction U$, and $\tilde\varphi[\nu Y\setminus U]$ is nowhere dense.
\end{theorem}

To prove Theorems~\ref{thm:main} and \ref{thm:main2}, and in general versions of the $\cWEP$, we use two main ingredients: first, we give a notion of local triviality for a unital $^*$-homomorphism $\Phi\colon C_\nu(X)\to C_\nu(X)$, and show that if $\Phi$ is surjective and locally trivial, then $\Phi$ is trivial. No set theory is involved here, i.e., this statement holds in $\ZFC$. Secondly, we show that indeed assuming $\OCA$ and $\MA$ all $^*$-homomorphisms between algebras of functions over Higson coronas are locally trivial.

The paper is structured as follows: in \S\ref{S.Triv}, after recalling some basic definitions, we introduce our main object of interest, trivial $^*$-homomorphisms for algebras of functions over Higson coronas, and prove a few properties about them. We also give the appropriate definition of the coarse weak Extension Principle. In \S\ref{S.localtoglobal} we give the notion of local triviality, and show that all locally trivial surjective $^*$-homomorphisms are trivial, while in \S\ref{S.locallifts} we prove how local triviality follows from our set theoretic hypotheses, after having set up the appropriate context for the lifting results of \cite{mckenney2018forcing} and \cite{TrivIsoMetric} to be applied. 

\subsection{Set theoretic preliminaries}\label{ss.ST}
In these notes, we use Todor\v cevi\'c's Open Colouring Axiom $\OCA$ and Martin's Axiom $\MA$. The combination of these axioms is consistent with the usual theory $\ZFC$, the Zermelo--Fraenkel theory together with the Axiom of Choice, without having to assume any large cardinal hypotheses\footnote{Indeed, every model of $\ZFC$ has a forcing extension which has the same $\omega_1$ and satisfies $\OCA$ and $\MA_{\aleph_1}$, see the sketch contained in \cite[\S2]{Velickovic.OCA1} resembling a technique already used in \cite{ARS}.}.

$\OCA$ was formalised in its current form by Todor\v cevi\'c in \cite{Todorcevic.PPIT}. The current version of this axiom has roots in the work of Baumgartner (\cite{Baum.Aleph1}), and it is a modification of several colouring axioms appearing in work of Abraham, Rubin, and Shelah (\cite{ARS}). As $\OCA$ will be used directly in the proof of Theorem~\ref{thm:local}: $\OCA$ is the assertion that for every separably metrizable space $\cX$ and an open $K_0\subseteq [\cX]^2$ either there exists an uncountable $\cY\subseteq \cX$ such that $[\cY]^2\subseteq K_0$ or there is a partition $\cX=\bigcup_n \cX_n$ such that each $[\cX_n]^2$ is disjoint from $K_0$. $\OCA$ contradicts $\CH$, and in fact it implies that $\mathfrak b = \omega_2$, $\mathfrak b$ being the bounding number (this will be defined properly when needed, see Claim~\ref{claim:uniformising}.)

Martin's Axiom at level $\aleph_1$, $\MA$, is the straightforward generalisation of the Baire category theorem for spaces with the countable chain condition, asserting that $\aleph_1$ many dense sets must intersect. (As we do not directly use $\MA$ in these notes, we will not define it.) $\MA$ has profound applications in many areas of set theory, especially set theoretic topology. $\MA$ obviously contradicts $\CH$. 

Veli\v ckovi\'c proved in \cite{Ve:OCA} that combining $\OCA$ and $\MA$ one has that all automorphisms of $\mathcal P(\bbN)/\Fin$ are trivial (this was known to be true under $\PFA$, see \cite{ShSte:PFA}). While this result was recently proved to hold under just $\OCA$ (\cite{TrivIso}), $\OCA$ and $\MA$ (or weakenings of their conjunction) were the running assumptions for almost all lifting results for quotient structures obtained in the last three decades\footnote{We conjecture that $\MA$ is not needed for most such results, but at moment we still need $\MA$ for the lifting theorems proved in \cite{mckenney2018forcing} and \cite{TrivIsoMetric}, which we will use.} (for example \cite{Fa:All, FarahBook2000, vignati2018rigidity}, see \cite{CoronaRigidity} for more details).

\subsection*{Acknowledgements}
The author is thankful to the participant of the SQuaRE project `Expanders, ghosts, and Roe algebras' held at the American Institute of Mathematics (AIM) from 2021 to 2024. The author is also thankful to the AIM for the support of the aforementioned SQuaRE project. The author is partially funded by the Institut Universitaire de France.

\section{Trivial $^*$-homomorphisms and the coarse weak Extension Principle}\label{S.Triv}

The main objects of study of these notes are $^*$-homomorphisms between Higson coronas. These often arise from (coarse) geometry-preserving maps between the underlying spaces, as in Proposition~\ref{prop:roetrivial}. 

\begin{definition}\label{def:maps}
Let $(X,d_X)$ and $(Y,d_Y)$ be metric spaces. A function $\varphi\colon Y\to X$ is
\begin{itemize}
\item \emph{proper} if the inverse image of any bounded set is bounded;
\item \emph{coarse} if for every $R>0$ there is $S>0$ such that $d_Y(y,y')\leq R$ implies $d_X(\varphi(y),\varphi(y'))\leq S$ ($\varphi$ maps close points to close points);
\item \emph{expanding} if for every $S>0$ there is $R>0$ such that if $d_Y(y,y')\geq R$ then $d_X(\varphi(y),\varphi(y'))\geq S$ ($\varphi$ maps far points to far points);
\item \emph{of cobounded range} if there is $R>0$ such that for every $x\in X$ there is $y\in Y$ with $d_X(x,\varphi(y))\leq R$ (the image of $\varphi$ is a net in $X$).
\end{itemize}
A coarse expanding map is a \emph{coarse embedding}, and a coarse embedding with cobounded range is a \emph{coarse equivalence}.
\end{definition}

 In the locally finite setting, as bounded and finite sets coincide, a map is proper if and only if it is finite-to-one.

We are ready to define our notion of triviality. If $X$ is a set, we denote by $\pi_X$ the quotient map $\pi_X\colon\ell_\infty(X)\to \ell_\infty(X)/c_0(X)$.

\begin{definition}\label{def:trivial}
Let $X$ and $Y$ be u.l.f.\ metric spaces. A unital $^*$-homomorphism $\Phi\colon C_\nu(X)\to C_\nu(Y)$ is \emph{trivial} if there is a function $\varphi\colon Y\to X$ such that for every $A\subseteq Y$ and every positive contraction $g\in C_h(X)$, we have that 
\[
\norm{\pi_Y(\chi_{A})\Phi(\pi_X(g))}=\norm{\pi_X(\chi_{\varphi[A]}g)}.
\]
In this case, we say that $\varphi$ induces $\Phi$.
\end{definition}
All unital $^*$-homomorphisms between algebras of functions over Higson coronas constructed using a proper coarse map as in Proposition~\ref{prop:roetrivial} are trivial. The following simplifies checking whether a given $^*$-homomorphism is trivial.
\begin{lemma}\label{lem:trivialbetter}
Let $X$ and $Y$ be u.l.f.\ metric spaces, and let $\Phi\colon C_\nu(X)\to C_\nu(Y)$ be a unital $^*$-homomorphism. Suppose that $\varphi\colon Y\to X$ is a proper function such that for all positive contractions $g\in C_h(X)$ and $A\subseteq Y$, 
\[
\text{ if } \norm{\pi_X(\chi_{\varphi[A]}g)}=1 \text{ then } \norm{\pi_Y(\chi_{A})\Phi(\pi_X(g))}\neq 0.
\] 
Then $\Phi$ is trivial and $\varphi$ induces it.
\end{lemma}
\begin{proof}
By contradiction, suppose there is a positive contraction $g\in C_h(X)$ and $A\subseteq Y$ such that $r:=\norm{\pi_Y(\chi_{A})\Phi(\pi_X(g))}\neq \norm{\pi_X(\chi_{\varphi[A]}g)}=:s$.
\begin{enumerate}
\item If $r<s$, a simple functional calculus gives a positive contraction $h$ such that
\[
\norm{\pi_X(\chi_{\varphi[A]}h)}=1 \text{ and } \norm{\pi_Y(\chi_{A})\Phi(\pi_X(h))}=0,
\]
leading to a contradiction. 
\item If $r>s$, let $\varepsilon=(s-r)/2$. As $r>0$, then $A$ is infinite, and by properness of $\varphi$ so is $\varphi[A]$. 
Let $h\in C_h(Y)$ be a positive contraction such that $\pi_Y(h)=\Phi(\pi_X(g))$, so that $\limsup_{y\in A}h(y)>r+\varepsilon$. Pick an infinite $B\subseteq A$ such that $h(y)>r+\varepsilon$ for all $y\in B$. Fix $g'=1-g$, so that $\pi_Y(1-h)=\Phi(\pi_Y(g'))$. Since $\varphi[B]$ is infinite, then $\norm{\pi_X(\chi_{\varphi[B]}g')}\geq 1-\norm{\pi_X(\chi_{\varphi[B]}g)}\geq 1-s$ (the last inequality is given by that $B\subseteq A$). On the other hand, for every $y\in B$ we have that $(1-h)(y)\leq 1-r-\varepsilon$, and so $\norm{\pi_Y(\chi_B)\Phi(\pi_X(1-h))}\leq 1-r-\varepsilon<1-s$. We have thus found a contraction such that the first case applies, a contradiction. \qedhere
\end{enumerate}
\end{proof}

We now study geometric properties of the map $\varphi$ obtained from our notion of triviality.
\begin{prop}\label{prop:trivialhoms}
Let $(X,d_X)$ and $(Y,d_Y)$ be u.l.f.\ metric spaces, and suppose that $\Phi\colon C_\nu(X)\to C_\nu(Y)$ is a trivial $^*$-homomorphism as induced by $\varphi\colon Y\to X$. Then
\begin{enumerate}
\item\label{trivialc1} $\varphi$ is coarse and proper;
\item\label{trivialc2} if $\Phi$ is surjective then $\varphi$ is expanding;
\item\label{trivialc3} if $\Phi$ is injective then $\varphi$ has cobounded range.
\end{enumerate}
\end{prop}
\begin{proof}
Let us prove \eqref{trivialc1}. We reason by contradiction. First assume $\varphi$ is not proper, so that there is an infinite $A\subseteq Y$ such that $\varphi[A]$ is finite. Since $\Phi$ is unital, then $\Phi(1)=1$. Since $\Phi$ is induced by $\varphi$, then
\[
1=\norm{\pi_Y(\chi_A)}=\norm{\pi_Y(\chi_{A})\Phi(\pi_X(1))}=\norm{\pi_X(\chi_{\varphi[A]}1)}=0,
\]
a contradiction.

Assume now $\varphi$ is not coarse, and let $(y_n)$ and $(y'_n)$ be sequences of distinct points such that $\sup_nd_Y(y_n,y'_n)<\infty$ yet $d_X(x_n,x_n')\geq n$, where $x_n=\varphi(y_n)$ and $x'_n=\varphi(y'_n)$. Since $\varphi$ is proper, by passing to a subsequence, we can assume that for all $n\neq m$ we have that
\[
\max\{d_X(x_n,x_m),d_X(x_n,x_m'),d_X(x_n',x_m')\}\geq m+n.
\]
Let $A=\{y_n\}$ and $B=\{y'_n\}$. Note that for every $g\in C_h(Y)$ we have that $\norm{\pi_Y(\chi_Ag)}=\norm{\pi_Y(\chi_Bg)}$. Consider now a function $f\in C_h(X)$ such that $f(x_n)=1$ and $f(x'_n)=0$ for all $n\in\mathbb N$. Since $\Phi(\pi_X(f))\in C_\nu(Y)$, the contradiction comes from that
\[
1=\norm{\pi_X(\chi_{\varphi[A]}f)}=\norm{\pi_Y(\chi_A)\Phi(\pi_X(f))}=\norm{\pi_Y(\chi_B)\Phi(\pi_X(f))}=\norm{\pi_X(\chi_{\varphi[B]}f)}=0.\]

The proof of \eqref{trivialc2} is almost the same. Suppose now that $\Phi$ is surjective, and assume that $\varphi$ is not expansive, so that there are two sequences of distinct points $(y_n)$ and $(y'_n)$ such that $d_Y(y_n,y_n')\geq n$ yet $\sup_nd_X(x_n,x_n')<\infty$ where $x_n=\varphi(y_n)$ and $x'_n=\varphi(y'_n)$. Again by going to a subsequence, using properness of $\varphi$, we can assume that for all $n\neq m$
\[
\max\{d_X(x_n,x_m),d_Y(y_n,y_m),d_Y(y_n,y_m'),d_Y(y'_n,y'_m)\}\geq m+n.
\]
Let $A=\{y'_n\}$ and $B=\{y'_n\}$, so that for all $f\in C_h(X)$ we have that $\norm{\pi_X(\chi_{\varphi[A]}f)}=\norm{\pi_X(\chi_{\varphi[B]}f)}$. Let $g\in C_h(Y)$ be a such that $g(y_n)=1$ and $g(y'_n)=0$ for all $n$, and pick $f\in C_h(X)$ such that $\Phi(\pi_X(f))=\pi_Y(g)$. Then
\[
1=\norm{\pi_Y(\chi_{A}g)}=\norm{\pi_X(\chi_{\varphi[A]}f)}=\norm{\pi_X(\chi_{\varphi[B]f)}}=\norm{\pi_Y(\chi_Bg)}=0.
\]
Let us now prove \eqref{trivialc3}. Assume this is not the case, and suppose that there is a sequence $(x_n)$ such that $d_X(x_n,\varphi[Y])\geq n$. Let $f\in C_h(X)$ be a contraction such that $f\restriction \varphi[Y]=0$ yet $f(x_n)=1$ for all $n$. Putting all of this together we get that
\[
0=\norm{\pi_X(\chi_{\varphi[Y]}f)}=\norm{\pi_Y(\chi_Y)\Phi(\pi_X(f))}=\norm{\Phi(\pi_X(f))}=\norm{\pi_X(f)}=1,
\] where the second to last equality follows from injectivity of $\Phi$. This is a contradiction.
\end{proof}
\begin{corollary}
Let $X$ and $Y$ be u.l.f.\ metric spaces. Let $\Phi\colon C_\nu(X)\to C_\nu(Y)$ be a trivial unital $^*$-homomorphism as induced by $\varphi\colon Y\to X$. Let $\tilde\varphi\colon \nu Y\to \nu X$ be the dual of $\Phi$. Then $\nu\varphi=\tilde\varphi$. If $\Phi$ is surjective, then $\varphi$ is a coarse embedding and if $\Phi$ is an isomorphism, $\varphi$ is a coarse equivalence.
\end{corollary}

We now give some examples of nontrivial $^*$-homomorphisms.

\begin{example}\label{example:nontrivial}
\begin{itemize}
\item Fix u.l.f.\ metric spaces $X$ and $Y$, and consider a maximal ideal $I\subseteq C_\nu(X)$. The quotient map $\pi_I\colon C_\nu(X)\to C_\nu(X)/I\cong\mathbb C$ composed with the inclusion $\mathbb C\cdot 1\subseteq C_\nu(Y)$, $\pi_I$ gives a nontrivial $^*$-homomorphism $C_\nu(X)\to C_\nu(Y)$.
\item The $^*$-homomorphism above has indeed very small image (just the canonical copy of $\mathbb C$), but we can do better. Consider $X=\{n^2\mid n\in\mathbb N\}$ with the metric induced from $\mathbb R$. In this case, $C_h(X)=\ell_\infty(X)$ and $C_\nu(X)=\ell_\infty(X)/c_0(X)\cong C(\beta\mathbb N\setminus\mathbb N)$, $\beta\mathbb N$ being the \v{C}ech--Stone compactification of $\mathbb N$. Using the nontrivial copy of $\beta\bbN\setminus \bbN$ inside $\beta\bbN\setminus \bbN$ constructed by Dow (\cite{dow2014non}), we get a nontrivial unital $^*$-homomorphism $C_\nu(X)\to C_\nu(X)$. Since Dow's construction, performed in $\ZFC$, gives a copy of $\beta\mathbb N\setminus\mathbb N$ inside itself which is nowhere dense, the kernel of its dual $^*$-homomorphism is an essential ideal; on the other hand, the (everywhere) nontrivial $^*$-homomorphism is surjective.
\item The following was originally remarked by Protasov in \cite{Protasov.Corona}. Suppose that $X$ is a u.l.f.\ metric space of asymptotic dimension zero. This is equivalent to that the metric on $X$ is coarsely equivalent to an ultrametric, or to that $\nu X$ is a zero dimensional topological space). It is a fairly instructive exercise to show that the Boolean algebra of clopen sets of $\nu X$ is countably saturated (in the language of model theory, see for example \cite[\S6.1]{CoronaRigidity}) and is atomless. Therefore, if one assumes the Continuum Hypothesis $\CH$, $C_\nu(X)$ is isomorphic to $\ell_\infty/c_0$, even though there $2^{\aleph_0}$ mutually non-coarsely equivalent spaces of asymptotic dimension zero. This gives rise to an abundance of nontrivial $^*$-isomorphisms. Using a similar model theoretic argument, recently Brian and Farah (\cite{BrianFarah}) showed the existence of $2^{\aleph_0}$ mutually non-coarsely equivalent spaces each of which has asymptotic dimension one and whose Higson coronas are all isomorphic, again if one assumes $\CH$.
\end{itemize}
\end{example}

As we have just seen, in $\ZFC$ it is possible to construct nontrivial surjective $^*$-homomorphisms between algebras of functions over Higson coronas, but all examples (unless one is willing to assume additional axioms as $\CH$) have a large kernel (an essential ideal). (For a unital $^*$-homomorphism between abelian \cstar-algebras, the kernel being an essential ideal corresponds precisely to the image of the dual map being nowhere dense.) We code this behaviour in the following definition.

\begin{definition}\label{def:cwep}
Let $X$ and $Y$ be u.l.f.\ metric spaces, and let $\Phi\colon C_\nu(X)\to C_\nu (Y)$ be a unital $^*$-homomorphism with dual $\varphi\colon\nu Y\to \nu X$. We say that the coarse weak Extension Principle holds for $\Phi$, and write $\cWEP(\Phi)$, if there exists $\tilde U\subseteq Y$ such that the corresponding $U\subseteq\nu Y$ is clopen, $\tilde\varphi[\nu Y\setminus U]$ is nowhere dense, and
\[
\Phi^U=\chi_{U}\Phi\chi_{U}\colon C_\nu(X)\to C_\nu(\tilde U)
\]
is trivial.

We say that the coarse weak Extension Principle holds, and write $\cWEP$, if $\cWEP(\Phi)$ holds for all unital $^*$-homomorphisms between algebras of functions over Higson coronas, and we say the the surjective coarse weak Extension Principle holds, and write $\cWEP(s)$, if $\cWEP(\Phi)$ holds for all unital surjective $^*$-homomorphisms between algebras of functions over Higson corona
\end{definition}

The coarse weak Extension Principle $\cWEP$ takes its name from the weak Extension Principle $\WEP$ as introduced by Farah in \cite{FarahBook2000} to study maps between \v{C}ech--Stone remainders of locally compact zero dimensional spaces, and later extended (and proved to follow from $\OCA$ and $\MA$) to all remainders of locally compact metrizable spaces in \cite{VY:wep}. The necessity of the set $U$ such that the image of $\nu Y\setminus U$ is nowhere dense is precisely to take in account Dow's $\ZFC$ example of a nontrivial copy of $\beta\mathbb N\setminus\mathbb N$ inside itself (Example~\ref{example:nontrivial}), which precisely shows that the strong version of the Extension Principle (see \cite[\S4.11]{FarahBook2000}) fails in $\ZFC$. Clearly, under the Continuum Hypothesis $\CH$ the $\cWEP$ fails, but, as we will show in Theorem~\ref{thm:OCAwep}, it is consistent that the the $\cWEP$ holds.

\section{Local triviality}\label{S.localtoglobal}
In this section, we introduce `locally trivial' $^*$-homomorphisms and show that local triviality implies triviality, at least when surjectivity is assumed. We start with a key definition.
\begin{definition}
Let $(X,d)$ be a u.l.f.\ metric space. A sequence $\bar X=(X_n)$ of pairwise disjoint finite subsets of $X$ is said to be sparse\footnote{This is not the usual definition of sparse sequence, which requires that $\lim_{m+n\to\infty}d(X_n,X_m)=\infty$. If $\bar X=(X_n)$ is a sequence of finite subsets of $X$ which is sparse according to this (weaker) definition, one can take finite unions of the $X_n$ and obtain a sequence which is sparse according to our definition. We do not require the sets $X_n$ to be necessarily nonempty.} if $d(X_n,X_m)>2\max\{n,m\}$ for all $n,m\in\mathbb N$.
\end{definition}

We record an easy lemma on functions supported on sparse sequences. This will be used in \S\ref{S.locallifts}. 

Recall that a linear function between \cstar-algebras is positive if it maps positive elements to positive elements, and order zero if it preserves orthogonality (see for example \cite[\S3.2.1]{Fa:STCstar}). 

\begin{lemma}\label{lem:orderzero}
Let $(X,d)$ be a u.l.f.\ metric space, and let $\bar X=(X_n)$ be a sequence of finite disjoint subsets of $X$. Let $g\in \ell_\infty(X)$ be supported on $\bigcup X_n$, and write $g=\sum g_n$ where $\supp(g_n)\subseteq X_n$ for all $n$. If $S\subseteq\bbN$, let 
\[
g_S=\sum_{n\in S} g_n.
\]
Then:
\begin{enumerate}
\item\label{o0c1} $g\in C_h(X)$ if and only if for every infinite $T\subseteq \bbN$ there is an infinite $S\subseteq T$ such that $g_S\in C_h(X)$, and 
\item\label{o0c2} if each $g_n$ is positive and contractive, the map $S\mapsto g_S$ induces an order zero positive contractive map $\ell_\infty\to \ell_\infty(X)$ which induces an order zero positive contractive map $\ell_\infty/c_0\to \ell_\infty(X)/C_0(X)$.
\end{enumerate}
\end{lemma}
\begin{proof}
We prove \eqref{o0c1}. Suppose that $g$ is slowly oscillating. Since $g_S=g\chi_{\bigcup_{n\in S}X_n}$, each $g_S$ is slowly oscillating. This also shows that if $S\subseteq T$ and $g_T$ is slowly oscillating so is $T$. This, together with the fact that for disjoint $S$ and $S'$ we have that $g_{S\cup S'}=g_S+g_{S'}$ shows that $\mathcal I_g=\{S\subseteq \bbN\mid g_S\text{ is slowly oscillating}\}$ is an ideal on $\bbN$.

Vice versa, suppose that $g$ is not slowly oscillating, and fix $\varepsilon>0$ and two infinite sequences $(x_n),(x'_n)\subseteq X$ such that $d(x_n,x_n')\leq R$ but $|g(x_n)-g(x_n')|\geq\varepsilon$. Since $g=\sum g_n$ where $\supp(g_n)\subseteq X_n$ for every $n$, and the supports of the functions $g_n$ are pairwise disjoint, we can assume there are two sequences $(k_n)$ and $(j_n)$ such that $g(x_n)=g_{k_n}(x_n)$ and $g(x_n')=g_{j_n}(x_n')$. By going (twice if needed) to a subsequence we can assume that $T=\{k_n\mid n\in\mathbb N\}\cup\{j_n\mid n\in\mathbb N\}$ belongs to the dense ideal $\mathcal I_g$. Since $g_T$ is slowly oscillating, then 
\[
\lim_n|g_T(x_n)-g_T(x'_n)|=0.
\]
As $g_T(x_n)=g_{k_n}(x_n)=g(x_n)$ and $g_{T}(x_n')=g_{j_n}(x_n')=g(x_n')$, we have a contradiction.

We now show \eqref{o0c2}. Consider the map sending $\chi_S\mapsto g_S$, $\chi_S$ being the characteristic function on $S$. This map preserves orthogonality, since the sets $X_n$ are pairwise disjoint, and it maps every projection in $\ell_\infty$ to a positive contraction in $\ell_\infty(X)$. Since $\ell_\infty$ is generated (as a \cstar-algebra) by its projections (\cite[3.1.11]{Fa:STCstar}), the association $\chi_S\mapsto g_S$ extends to an order zero positive contractive map $\ell_\infty\to\ell_\infty(X)$. Furthermore, if $S$ is finite, then $g_S\in C_0(X)$, therefore the image of $c_0$ is contained in $C_0(X)$. This concludes the proof. 
\end{proof}

$\mathbb D$ denotes the closed unit disk. If $A$ is a set, $\mathbb D^A$ corresponds to contractions in $\ell_\infty(A)$. If $A\subseteq B$, we see $\mathbb D^A$ as embedded in $\mathbb D^B$, extending $f\in \mathbb D^A$ to an element of $\mathbb D^B$ by setting $f\restriction B\setminus A=0$. If $A$ is empty, $\mathbb D^A$ is the singleton whose element is the empty function.

\begin{definition}\label{def:productformspseq}
Let $X$ and $Y$ be sets, and let $\bar X=(X_n)$ be a sequence of pairwise disjoint finite subsets of $X$.
 
 A function $\Lambda\colon \prod_n\mathbb D^{X_n}\to \mathbb D^Y$ is \emph{of product form} on $\bar X$ if there exist pairwise disjoint finite sets $Y_n\subseteq Y$ and functions $\alpha_n\colon\mathbb D^{X_n}\to\mathbb D^{Y_n}$ such that $\Lambda=\sum\alpha_n$, that is, 
\[
\Lambda((g_n))=\sum_n\alpha_n(g_n)\ \ \text{ for all }(g_n)\in\prod_n\mathbb D^{X_n}.
\]
If $X$ and $Y$ are u.l.f.\ metric spaces, a $^*$-homomorphism $\Phi\colon C_\nu(X)\to C_\nu(Y)$ is of product form on $\bar X$ if there is $\Lambda$, a map of product form on $\bar X$, such that for all $g\in C_h(X)$ with $\supp(g)\subseteq\bigcup X_n$ we have
\[
\pi_Y(\Lambda(g))=\Phi(\pi_X(g)).
\]
\end{definition}
$^*$-homomorphisms of product form on a sparse sequence $\bar X$ give forms of `local triviality'. The main focus of this section is to show that local triviality suffices for our purposes, as proved in \S\ref{ss:localtoglobal} and specifically Theorem~\ref{thm:localtoglobal}. Before doing so, we give a few properties of $^*$-homomorphisms of product form. For this, we shall need the following filtration of $C_h(X)$

Let $(X,d)$ be a u.l.f.\ metric space. If $Z\subseteq X$ and $k\in\mathbb N$, we write $Z^{+k}$ for the ball of radius $k$ around $Z$, that is,
\[
Z^{+k}=\{x\in X\mid d(x,Z)\leq k\}.
\]
Let $\bar X=(X_n)$ be a sparse sequence in $X$, and define, for $n,m\in\bbN$,
\begin{align*}
M_{n,m,\bar X}&:=\{g\in\mathbb D^{(X_n)^{+n}}\mid \supp(g)\subseteq X_n\text{ and } \\
&\forall x,x'\in (X_n)^{+n} \,\,(d(x,x')\leq m\Rightarrow |g(x)-g(x')|\leq \frac{1}{m+1})\}.
\end{align*}
For $f\in \NN^\NN$, define
\[
M_{f,\bar X}:=\prod_n M_{n,f(n),\bar X}.
\]
When $\bar X$ is clear from the context, we simply write $M_{n,m}$ and $M_f$ for $M_{n,m,\bar X}$ and $M_{f,\bar X}$. $\N^{\N\uparrow}$ is the set of all functions $f\colon\N\to\N$ such that $f(n)\to\infty$ as $n\to\infty$.
\begin{prop}\label{prop:Mf}
Let $(X,d)$ be a u.l.f.\ metric space and let $\bar X=(X_n)$ be a sparse sequence in $X$. Then for all naturals $n$, $m$ and $m'$ and $f,f'\in\N^\N$ we have that:
\begin{enumerate}
\item\label{propMf0} If $m\leq m'$, then $M_{n,m'}\subseteq M_{n,m}$. If $f'\leq f$, then $M_{f}\subseteq M_{f'}$.
\item\label{propMf3} For every contraction $g\in C_h(X)$ with $\supp(g)\subseteq\bigcup X_n$ there is $f\in\N^{\N\uparrow}$ such that $g\in M_f$.
\item\label{propMf4} If $f\in\N^{\N\uparrow}$ then $M_{f}\subseteq C_h(X)$.
\end{enumerate}
\end{prop}
\begin{proof}
\eqref{propMf0}: Let $g\in M_{n,m'}$, and pick $x,x'\in (X_n)^{+n}$ with $d(x,x')\leq m\leq m'$. Since $g\in M_{n,m'}$, then $|g(x)-g(x')|\leq\frac{1}{m'+1}\leq\frac{1}{m+1}$, hence $g\in M_{n,m}$. The second part of \eqref{propMf0} follows from the definition.


\eqref{propMf3}: Pick a contraction $g\in C_h(X)$ with $\supp(g)\subseteq\bigcup X_n$. We define a strictly increasing sequence $n_i$, for $i\in\N$. Let $n_0=0$. If $n_{i-1}$ has been defined, let $n_i>n_{i-1}$ be the minimum natural such that if $d(x,x')\leq i$, then $|g(x)-g(x')|\leq\frac{1}{i+1}$ for all $x,x'\in\bigcup_{j\geq n_i}(X_j)^{+j}$. This natural exists as $g$ is slowly oscillating and the sequence $\bar X$ is sparse. Let now, for $n\in\N$, $f(n)=i$ where $n_i\leq n<n_{i+1}$, so that $f\in\N^{\N\uparrow}$. We want to show that $g\in M_{f}$, so fix $n$ and $x,x'\in (X_n)^{+n}$ with $d(x,x')\leq f(n)$. Since $n\geq n_{f(n)}$, then by construction $|g(x)-g(x')|\leq\frac{1}{f(n)+1}$, hence $g\restriction X_n\in M_{n,f(n)}$.

\eqref{propMf4}: Fix $f\in\N^{\N\uparrow}$ and let $g\in M_f$. Fix $R,\varepsilon>0$. To prove that $g$ is slowly oscillating, we have to find a finite $Z\subseteq X$ such that if $x,x'\notin Z$ are such that $d(x,x')\leq R$ then $|g(x)-g(x')|\leq\varepsilon$. Let $m>R$ be large enough so that $f(n)\geq R$ and $\frac{1}{f(n)+1}<\varepsilon$ for every $n\geq m$, and let $Z=\bigcup_{i\leq m}X_i$. Since each $X_i$ is finite, so is $Z$. We claim $Z$ is as required: Fix then $x,x'\notin Z$ with $d(x,x')\leq R$. If $x,x'\notin\bigcup X_n$ then $g(x)=g(x')=0$, so there is nothing to prove. Otherwise, assume $x\in X_n$. Since $x\notin Z$ then $n>m$, so, as $d(x,x')<R<n$, we have that $x'\in (X_n)^{+n}$. Since $R\leq f(n)$, we have $d(x,x')\leq R\leq f(n)$, and therefore $|g(x)-g(x')|\leq \frac{1}{f(n)+1}\leq \varepsilon$. This concludes the proof.
\end{proof}

The following proposition shows that liftings of product form do not only lift slowly oscillating functions, but also behave approximately on $M_{f}$ for a large enough constant $f$. 

\begin{prop}\label{prop:them1}
Let $(X,d_X)$ and $(Y,d_Y)$ be u.l.f.\ metric spaces. Let $\Phi\colon C_\nu(X)\to C_\nu(Y)$ be a $^*$-homomorphism. Suppose that $\Phi$ is of product form on a sparse sequence $\bar X=(X_n)$, as witnessed by $\Lambda=\sum\alpha_n$ where $\alpha_n\colon \mathbb D^{X_n}\to\mathbb D^{Y_n}$. 
Then for all $\varepsilon>0$ there are naturals $m$ and $\bar n$ such that for all $n\geq \bar n$ and $g_n,g'_n\in M_{n,m}$ we have that
\begin{enumerate}
\item\label{clause3} $\norm{\alpha_n(g_n+g_n')-\alpha_n(g_n)-\alpha_n(g_n')}\leq\varepsilon$, 
\item\label{clause2} $\norm{\alpha_{n}(g_ng_n')-\alpha_n(g_n)\alpha_n(g_n')}\leq\varepsilon$,
\item\label{clause1} if $\Phi$ is injective, then $|\norm{\alpha_{n}(g_n)}-\norm{g_n}|\leq\varepsilon$.
\end{enumerate}
In particular, if $g\in\prod_nM_{n,m}$ and $g'\in C_h(X)$ is such that $\supp(g')\subseteq\bigcup X_n$, then
\[
\norm{\pi_X(\Lambda(g+g')-\Lambda(g)-\Lambda(g'))}, \norm{\pi_X(\Lambda(gg')-\Lambda(g)\Lambda(g'))}\leq \varepsilon.
\]
\end{prop}
\begin{proof}
We find naturals $m_i$ and $\bar n_i$ such that the pair $(m_i,\bar n_i)$ satisfies clause (i), for $i=1,2,3$, and then take their maxima. Since the proofs are precisely the same, we only show the details for \eqref{clause3}. 

Suppose \eqref{clause3} fails, meaning that for all $m$ and $n$ there is $n'\geq n$ and $g,g'\in M_{n',m}$ such that $\norm{\alpha_{n'}(gg')-\alpha_{n'}(g)\alpha_{n'}(g')}>\varepsilon$. We can thus construct two strictly increasing sequences $(m_k)$ and $(n_k)$ and contractions $g_k,g'_k\in M_{n_k,m_k}$ such that for all $k$
\[
\norm{\alpha_{n_k}(g_k+g'_k)-\alpha_{n_k}(g_k)-\alpha_{n_k}(g'_k)}>\varepsilon.
\]
Let $g=\sum g_k$ and $g'=\sum g'_k$, and note that $g,g'\in C_h(X)$, by \eqref{propMf4} of Proposition~\ref{prop:Mf}. Since $\Lambda$ is a lift of a linear map on slowly oscillating functions supported on $\bigcup X_n$, we have that
\[
\Lambda(g+g')-\Lambda(g)-\Lambda(g')=\sum\alpha_{n_k}(g_k+g'_k)-\sum\alpha_{n_k}(g_k)-\sum\alpha_{n_k}(g'_k)\in C_0(X).
\]
The contradiction now arises from that \begin{align*}
0&=\norm{\pi_Y(\Lambda(g+g')-\Lambda(g)-\Lambda(g'))}\\&=\limsup_k\norm{\alpha_{n_k}(g_k+g'_k)-\alpha_{n_k}(g_k)-\alpha_{n_k}(g'_k)}>\varepsilon.
\end{align*}
This ends the proof for clause \eqref{clause3}.

We now prove the last statement. Fix $g$ and $g'$ as in the hypotheses, and write $g=\sum g_n$ and $g'=g'_n$ where $\supp(g_n),\supp(g'_n)\subseteq X_n$ for all $n$. Since $g'\in C_h(X)$, by \eqref{propMf3} of Proposition~\ref{prop:Mf} for all sufficiently large $n$ we have that $g'_n\in M_{n,m}$. By how $m$ was chosen, after a given natural we have that $\norm{\alpha_{n}(g_n+g'_n)-\alpha_{n}(g_n)-\alpha_{n}(g'_n)}\leq\varepsilon$. Noticing that 
\[
\norm{\pi_Y(\Lambda(g+g')-\Lambda(g)-\Lambda(g'))}=\limsup_n\norm{\alpha_{n}(g_n+g'_n)-\alpha_{n}(g_n)-\alpha_{n}(g'_n)}
\] 
ends the proof. The proof for products is precisely the same.
\end{proof}
\begin{remark}\label{remark:subseq}
Note that since $M_{n,m}\supseteq M_{n,m'}$ whenever $m\leq m'$, if $m$ satisfies the thesis of proposition so does any $m'\geq m$. Further, the choice of $m$ remains the same if one takes a subsequence of $\bar X$ and applies the same map $\Lambda$ (only this time restricted to functions supported on the subsequence).
\end{remark}
We now show approximate coherence for liftings of product form.
\begin{prop}\label{prop:them2}
Let $(X,d_X)$ and $(Y,d_Y)$ be u.l.f.\ metric spaces. Let $\Phi\colon C_\nu(X)\to C_\nu(Y)$ be a $^*$-homomorphism. Suppose that $\Phi$ is of product form on the sparse sequences $\bar X=(X_n)$ and $\bar W=(W_n)$, as witnessed by $\Lambda_{\bar X}=\sum\alpha_n$ where $\alpha_n\colon \mathbb D^{X_n}\to\mathbb D^{Y_n}$ and $\Lambda_{\bar W}=\sum\beta_n$ where $\beta_n\colon \mathbb D^{W_n}\to\mathbb D^{Z_n}$. Then for every $\varepsilon>0$ there are $m$ and $\bar n$ such that if $g\in M_{n,m,\bar X}\cap M_{n,m,\bar W}$ for $n\geq \bar n$ then 
\[
\norm{\alpha_n(g)-\beta_n(g)}<\varepsilon.
\]
\end{prop}
\begin{proof}
We argue by contradiction and suppose that for every $m$ there is arbitrarily large $n$ and $g\in M_{m,n,\bar X}\cap M_{m,n,\bar W}$ such that $\norm{\alpha_n(g)-\beta_n(g)}\geq\varepsilon$. We construct two increasing sequences of naturals $(m_i)$ and $(n_i)$, and functions $g_i\in M_{m_i,n_i,\bar X}\cap M_{m_i,n_i,\bar W}$ such that for each $i$ we have that
\[
\norm{\alpha_{n_i}(g_i)-\beta_{n_i}(g_i)}\geq\varepsilon.
\]
By passing to a subsequence, as all sets $X_n$, $W_n$, $Y_n$ and $Z_n$ are finite, we can assume that for all $i\neq j$ we have that $X_{n_i}\cap W_{n_j}=\emptyset=Y_{n_i}\cap Z_{n_j}$. Let $g=\sum g_i$. Since each $g_i$ is a contraction, $g\in\ell_\infty(X)$, and its support is contained in $(\bigcup X_n)\cap (\bigcup W_n)$. Since $Y_{n_i}\cap Z_{n_j}=\emptyset$ for all $i\neq j$, then
\[
\norm{\pi_Y(\Lambda_{\bar X}(g)-\Lambda_{\bar W}(g))}=\limsup_i\norm{\alpha_{n_i}(g_i)-\beta_{n_i}(g_i)}\geq\varepsilon.
\]
Since $m_i\to\infty$, clause \ref{propMf4} of Proposition~\ref{prop:Mf} gives that $g\in C_h(X)$. Since $\supp(g)\subseteq(\bigcup X_n)\cap (\bigcup W_n)$, $\Lambda_{\bar X}$ and $\Lambda_{\bar W}$ are lifting for $\Phi$ on $g$, the contradiction comes from that
\[
0=\norm{\pi_Y(\Lambda_{\bar X}(g)-\Lambda_{\bar W}(g))}\geq\varepsilon>0.\qedhere
\]
\end{proof}
\subsection{From local triviality to triviality}\label{ss:localtoglobal}
We can now state and prove the main result of this section, showing that locally trivial $^*$-homomorphisms are indeed trivial (Theorem~\ref{thm:localtoglobal}). First, a definition.

\begin{definition}\label{def:markers}
Let $(X,d)$ be a u.l.f.\ metric space and let $x\in X$. If $m\in\mathbb N$, the function $g_{x,m}\colon X\to [0,1]$ is defined as 
\[
g_{x,m}(x')=\max \{0,1-\frac{d(x',x)}{(m+1)^2}\}.
\]
\end{definition}
We call these `marker functions'. 
If $Z\subseteq X$ and $k\in\mathbb N$, we let
\[
\mathrm{int}(Z,k)=\{x\mid d(x,X\setminus Z)\geq k\}.
\]
The proof of the following is a simple verification.
\begin{lemma}\label{lem:markers1}
Let $(X,d)$ be a u.l.f.\ metric space and let $\bar X=(X_n)$ be a sequence of disjoint finite subsets of $X$. Let $m\in\mathbb N$. If $x\in\mathrm{int}(X_n,(m+1)^2)$, then $g_{x,m}\in M_{n,m}$.\qed
\end{lemma}


\begin{definition}\label{def:sparsesequences}
Let $(X,d)$ be a u.l.f.\ metric space, and fix $x_0\in X$. Define
\[
X^{\mathrm{e}}_n=\{x\mid 2^{6n+1}\leq d(x,x_0)\leq 2^{6n+5}\}\text{ and }X_n^{\mathrm{o}}=\{x\mid 2^{6n+4}\leq d(x,x_0)\leq 2^{6n+8}\},
\] 
and let
\[
\bar X^{\mathrm{e}}=(X^{\mathrm{e}}_n) \text{ and }\bar X^{\mathrm{o}}=(X^{\mathrm{o}}_n).
\]
\end{definition}
Both $\bar X^{\mathrm{e}}$ and $\bar X^{\mathrm{o}}$ are sparse sequences. Note it is possible that some of the $X_n^i$s, for $i\in\{\mathrm e,\mathrm o\}$, are empty.
\begin{lemma}\label{lem:decomposition}
Let $(X,d)$ be a u.l.f.\ metric space, and let $x_0\in X$. Let $\bar X^\mathrm e$ and $\bar X^\mathrm o$ be defined as above. For any $g\in C_h(X)$ there are functions $g_{\mathrm e}$ and $g_{\mathrm o}$ in $C_h(X)$ such that $\supp(g_i)\subseteq \bigcup X^{i}_n$, for $i\in\{\mathrm e,\mathrm o\}$, and withe property that $g-g_{\mathrm e}-g_\mathrm o\in C_0(X)$. Moreover, if $g$ is a contraction, we can assume that if $x\in X_n^i$ is such that $|g_i(x)|\geq 1/2$ then $x\in \mathrm{int}(X_n^i,n)$.
\end{lemma}
\begin{proof}
Fix $n$, and define 
\[
h_{\mathrm e,n}(x)=\max \{0,1-\frac{d(x,\mathrm{int}(X_n^\mathrm e,2n))}{2n}\}.
\]
$h_{\mathrm e,n}$ is $1$ on $\mathrm{int}(X_n^\mathrm e,2n)$, and $0$ outside of $X_n^\mathrm e$. 

Fix $i\in\{\mathrm e,\mathrm o\}$. If $n\neq m$ we have $d(X_n^i,X_m^i)>m+n$, the functions $h_{i,n}$, for $n\in\mathbb N$, are pairwise orthogonal. Moreover, these get more and more slowly oscillating as $n\to\infty$, and therefore $h_i=\sum h_{i,n}$ is slowly oscillating.

Note that if $h_{\mathrm e,n}(x)\geq 1/2$ then $x\in \mathrm{int}(X_n^\mathrm e,n)$. Let $h_\mathrm e=\sum_n h_{\mathrm e,n}$ and $h_\mathrm o=1-h_\mathrm e$. It is routine to show that $h_\mathrm o$ is supported on $\bigcup X_n^\mathrm o$ and that if $x\in X_n^\mathrm o$ is such that $h_\mathrm o(x)\geq 1/2$ then $x\in\mathrm{int}(X_n^\mathrm o,n)$. For every contraction $g\in C_h(X)$ the decomposition $g=h_\mathrm e g+h_\mathrm og$ is as required.
\end{proof}

\begin{theorem}\label{thm:localtoglobal}
Let $(X,d_X)$ and $(Y,d_Y)$ be u.l.f.\ metric spaces, and suppose that $\Phi\colon C_\nu(X)\to C_\nu(Y)$ is a surjective $^*$-homomorphism. If $\Phi$ is of product form on both $\bar X^{\mathrm{e}}$ and $\bar X^{\mathrm{o}}$, then $\Phi$ is trivial.
\end{theorem}

If $W\subseteq X$, $\alpha\colon \mathbb D^W\to\mathbb D^Y$ is a map, $\varepsilon>0$ and $g\in \mathbb D^W$, we let 
\[
Y_{\alpha,g,\varepsilon}=\{y\in Y\mid |\alpha(g)(y)|>\varepsilon\}.
\]
The set $Y_{\alpha,g,\varepsilon}$ measures in some way the essential support of $\alpha(g)$. These are going to be used to `correct' the sets $Y_n$ arising from `local triviality'.
\begin{proof}[Proof of Theorem~\ref{thm:localtoglobal}]
For $i\in\{\mathrm e,\mathrm o\}$, we let $(Y^{i}_n)$ be a sequence of disjoint finite subsets of $Y$ and let $\Lambda_i=\sum\alpha^i_n$, where $\alpha^i_n\colon \mathbb D^{X_n^i}\to \mathbb D^{Y^i_n}$, be witnessing that $\Phi$ is of product form on $\bar X^i$. In particular, $\Lambda_i$ lifts $\Phi$ on slowly oscillating functions supported on $\bigcup X^i_n$.

The main idea of the proof is to use the marker functions $g_{x,m}$ for a suitable natural $m$, and to send $y\in Y$ to some (any) $x$ such that $\max\{\alpha^\mathrm e_n(g_{x,m}),\alpha^\mathrm o_n(g_{x,m})\}$ is large on $y$. We thus find a natural $m$ which is large enough for our purposes, prove a few properties of the sets $Y_{\alpha^i_n,g_{x,m},\delta}$, and then show that it is indeed possible to define a suitable $\varphi\colon Y\to X$ which induces $\Phi$ as in Definition~\ref{def:trivial}.
 
Let $\varepsilon=1/24$, and fix naturals $m$ and $\bar n$ such that the conclusion of Proposition~\ref{prop:them1} is satisfied for $\Phi$, $\varepsilon$, and both pairs $(\bar X^{\mathrm{e}},\Lambda_{\mathrm{e}})$ and $(\bar X^\mathrm{o},\Lambda_{\mathrm{o}})$, and the conclusion of Proposition~\ref{prop:them2} is satisfied for $\Phi$, $\varepsilon$, and the pair $\{(\bar X^\mathrm{e},\Lambda_{\mathrm{e}}),(\bar X^\mathrm{o},\Lambda_{\mathrm{o}})\}$. In particular for all $n\geq \bar n$ and $i\in\{\mathrm e,\mathrm o\}$, $\alpha^i_n$ is approximately linear and approximately multiplicative on $M_{n,m,\bar X^i}$, and $\alpha^\mathrm e_n$ and $\alpha^\mathrm o_n$ approximately agree on $M_{n,m,\bar X^\mathrm e}\cap M_{n,m,\bar X^\mathrm o}$. By Remark~\ref{remark:subseq}, we can assume that $\bar n\geq (m+1)^2$ and that $1/\varepsilon^2<m$.

%

\begin{claim}\label{claim:close}
There is $R>0$ such that for all $i,j\in\{\mathrm e,\mathrm o\}$ and large enough naturals $k$ and $k'$, if $y\in Y_{\alpha^i_k,g_{x,m},1/4}\cap Y_{\alpha^j_{k'},g_{x',m},1/4}$ for $x\in \mathrm{int}(X_k^i,k)$ and $x'\in \mathrm{int}(X_{k'}^j,k')$, then $d_X(x,x')\leq R$.
\end{claim}
\begin{proof}
Suppose that this is not the case. The negation of the thesis gives $i,j\in\{\mathrm e,\mathrm o\}$ and, for every $n$, two naturals $k_n$ and $k'_n$ and points $x_n,x'_n\in X$ with $x_n\in \mathrm{int}(X_{k_n}^i,k_n)$ and $x'_n\in \mathrm{int}(X_{k'_n}^j,k'_n)$ and $y_n\in Y$ such that 
\[
y_n\in Y_{\alpha^i_{k_n},g_{x_n,m},1/4}\cap Y_{\alpha^j_{k_n'},g_{x'_n,m},1/4}
\]
where $d_X(x_n,x_n')\geq n$. We can assume that $(k_n)$ and $(k'_n)$ are strictly increasing. Let $g$ and $g'$ be slowly oscillating orthogonal positive contractions such that $g(x_n)=1=g'(x_n')$ for all $n$. Since $d_X(x_n,X\setminus X_{k_n}^i)\to\infty$ as $n\to\infty$, we can assume that $\supp(g)\subseteq\bigcup X_n^i$ and write $g=\sum g_n$ where $\supp(g_n)\subseteq X_n^i$ for all $n$. Likewise, we write $g'=\sum g'_n$ where $\supp(g'_n)\subseteq X_n^j$ for all $n$. Since $g$ is slowly oscillating and $g(x_n)=1$, we have that $\limsup\norm{g_{k_n}g_{x_n,m}-g_{x_n,m}}=0$. Likewise $\limsup\norm{g'_{k'_n}g_{x'_n,m}-g_{x'_n,m}}=0$. By our choice of $m$ and $\varepsilon$, we have that 
\[
\limsup\norm{\alpha_n^i(g_{k_n})\alpha_n^i(g_{x_n,m})-\alpha_n^i(g_{x_n,m})}\leq 3\varepsilon
\] 
and 
\[ 
\limsup\norm{\alpha_n^j(g'_{k'_n})\alpha_n^j(g_{x'_n,m})-\alpha_n^j(g_{x'_n,m})}\leq 3\varepsilon.
\]
Since for every $n$ we have that $|\alpha_{k_n}^i(g_{x_n,m})(y)|, |\alpha_{k_n}^j(g_{x_n,m})(y)|>1/4$, we have that
\[
\min \{\limsup|\alpha_{k_n}^i(g_{k_n})(y)|, \limsup|\alpha_{k'_n}^j(g'_{k'_n})(y)|\} >0,
\]
 which implies that $\Lambda_i(g)$ and $\Lambda_j(g)$ do not have orthogonal images in $\ell_\infty(Y)/c_0(Y)$. As these lift $\Phi(\pi_X(g))$ and $\Phi(\pi_X(g'))$ respectively, and $gg'=0$, we have a contradiction.
\end{proof}

 
 Define, for $n\geq\bar n$ and $i\in\{\mathrm e,\mathrm o\}$,
\[
W^i_n= \bigcup \{Y_{\alpha^i_n,g_{x,m},1/4}\mid x\in \mathrm{int}(X_n^i,n)\},
\]
and let
\[
Y'=\bigcup_{n\geq\bar n}W_n^\mathrm e\cup W_n^\mathrm o.
\]
\begin{claim}\label{claim:cob}
$Y'$ is cobounded in $Y$.
\end{claim}
\begin{proof}
We argue by contradiction and assume there is a sequence $(y_n)\subseteq Y$ such that $d_Y(y_n,Y')\geq n$ for all $n$. Let $g\in C_h(Y)$ be a positive contraction such that $g(y_n)=1$ and $g\restriction Y'=0$. Since $\Phi$ is surjective, there is a positive contraction $h\in C_h(X)$ such that $\Phi(\pi_X(h))=\pi_Y(g)$. Let $h_\mathrm e=\sum h_{\mathrm e,n}$ and $h_\mathrm o=\sum h_{\mathrm o,n}$ be positive slowly oscillating contractions where, for $i\in\{\mathrm e,\mathrm o\}$ and $n\in\mathbb N$, $\supp(h_{ i,n})\subseteq X_n^{i}$ and such that we have $h-h_\mathrm e-h_\mathrm o\in C_0(X)$. These are given by Lemma~\ref{lem:decomposition}. Since $\Lambda_i$ lifts $\Phi$ on functions supported on $\bigcup X_n^i$, we have that $\Lambda_\mathrm e(h_\mathrm e)+\Lambda_\mathrm o(h_\mathrm o)-g\in C_0(Y)$. Since $\norm{\pi_Y(g)}=1$ and 
\[
\pi_Y(g)=\pi_Y(\Lambda_\mathrm e(h_\mathrm e))+\pi_Y(\Lambda_\mathrm o(h_\mathrm o)),
\]
there is $i\in \{\mathrm e,\mathrm o\}$ such that $\norm{\pi_Y(\Lambda_i(h_i))}\geq 1/2$. Assume $i=\mathrm e$. (The case where $\pi_Y(\Lambda_\mathrm o(h_\mathrm o))$ has large norm is treated in the same exact way.). Since $\Lambda_\mathrm e=\sum\alpha_n^\mathrm e$, we have that 
\[
\liminf_n\norm{\alpha^\mathrm e_n(h_{\mathrm e,n})}=\norm{\Lambda_\mathrm e(h_\mathrm e)}\geq 1/2.
\]
Since $\Phi(\pi_X(h_\mathrm e))\leq\pi_Y(g)$ and $\Lambda_\mathrm e(h_\mathrm e)$ lifts $\Phi(\pi_X(h_\mathrm e))$, we have that $|\Lambda_\mathrm e(h_\mathrm e)\restriction W_n^\mathrm e|\to 0$ as $n\to\infty$, and so $\limsup_n|\alpha^\mathrm e_n(h_\mathrm e)(W_n^\mathrm{e})|=0$. 
Since $\norm{\pi_X(h_\mathrm e)}\geq \norm{\pi_Y(\Lambda_\mathrm e(h_\mathrm e))}\geq 1/2$, we can find an increasing sequence of naturals $(k_n)$ with $k_0\geq (m+1)^2$ and points $x_n\in X^\mathrm e_{k_n}$ such that $h_{\mathrm e,k_n}(x_n)\geq 1/2-\varepsilon$. By Lemma~\ref{lem:decomposition} we have that $x_n\in \mathrm{int}(X_{k_n}^\mathrm e,k_n)$ and so $g_{x_n,m}\in M_{k_n,m,\bar X^{\mathrm e}}$. In particular, by our choice of $m$,
\[
\norm{\alpha_{k_n}^\mathrm e(h_{\mathrm e,k_n}) \alpha_{k_n}^\mathrm e(g_{x_n,m})}\geq \norm{h_{\mathrm e,k_n}g_{x_n,m}}-2\varepsilon \geq 1/2-3\varepsilon>1/3,
\]
which implies there is $y$ such that $|\alpha_{k_n}^\mathrm e(h_{\mathrm e,k_n})(y)|, | \alpha_{k_n}^\mathrm e(g_{x_n,m})(y)|>1/4$. Since $y\in Y_{k_n}^{\mathrm e}\subseteq Y'$ and $n$ is arbitrary, this contradicts that $\limsup_n|\alpha^\mathrm e_n(h_\mathrm e)(W_n^\mathrm{e})|=0$.
\end{proof}

We are ready to define the desired map $\varphi\colon Y\to X$. Using Claim~\ref{claim:cob}, we can fix $M$ such that for every $y\in Y$ there is $y'\in Y'$ with $d_Y(y,y')\leq M$. For $i\in\{\mathrm{e},\mathrm o\}$ define $\varphi^i\colon \bigcup W_n^i \to X$ by setting, for $y\in W_n^i$,
\[
\varphi^i(y)= x \text{ where } x \text{ is any point such that } y\in Y_{\alpha_n,g_{x,m},1/4}.
\]
Such $x$ exists as $W_n^i= \bigcup \{Y_{\alpha^i_n,g_{x,m},1/4}\mid x\in \mathrm{int}(X_n^i,n)\}$.

By Claim~\ref{claim:close}, $\varphi^\mathrm e$ and $\varphi^\mathrm o$ are close on the intersection of their domains, and therefore the map $\varphi$ defined by
\[
\varphi(y)=\begin{cases}\varphi^\mathrm e(y) & \text{ if } y\in \bigcup W_n^\mathrm e\\
\varphi^\mathrm o(y) &\text{ if } y\in Y'\setminus \bigcup W_n^\mathrm e
\end{cases}
\]
has domain $Y'$, codomain $X$, and it is close to both $\varphi^\mathrm e$ and $\varphi^\mathrm o$ on their respective domains. We extend $\varphi$ to $Y$, by sending $y\in Y\setminus Y'$ to $\varphi(y')$ where $y'\in Y'$ is such that $d_Y(y,y')\leq M$. Since, for $i\in\{\mathrm e,\mathrm o\}$, each $W_n^i$ is finite, $\varphi$ is proper.

\begin{claim}
The function $\varphi$ induces $\Phi$ as in Definition~\ref{def:trivial}. 
\end{claim}
\begin{proof}
Suppose this is not the case. By Lemma~\ref{lem:trivialbetter} we can find a positive contraction $g\in C_h(X)$ and $A\subseteq Y$ such that 
\[
\norm{\pi_X(\chi_{\varphi[A]}g)}=1 \text{ but } \norm{\pi_Y(\chi_{A})\Phi(\pi_X(g))}=0.
\]
Without loss of generality we can assume that $A\subseteq Y'$, as $Y'$ is cobounded in $Y$. Let $(y_k)\subseteq A$ be a sequence of distinct points such that $g(x_k)\to 1$ as $k\to\infty$, where $x_k=\varphi(y_k)$. Since, modulo finite, $X$ equals $\bigcup_n\mathrm{int}(X_n^\mathrm e,n)\cup \mathrm{int}(X_n^\mathrm o,n)$ we can assume, by eventually passing to a subsequence, that there is an infinite increasing sequence $(n_k)$ such that $x_k\in\mathrm{int}(X_{n_k}^\mathrm e,n_k)$ for some $k$. (The other case, where we have to deal with sets of the form $X_n^\mathrm{o}$, is treated in the same exact way noticing that $\varphi^\mathrm e$ and $\varphi^o$ are close on their common domain.) By writing $g=g_\mathrm e+g_\mathrm o$ as in Lemma~\ref{lem:decomposition}, and possibly applying functional calculus, we can assume that $\supp(g)\subseteq\bigcup X_n^\mathrm e$. Write $g=\sum g_n$ where $\supp(g_n)\subseteq X_n^\mathrm e$ for all $n$. Since $\norm{\pi_Y(\chi_{A})\Phi(\pi_X(g))}=0$, $\lim_k|\alpha_{n_k}^\mathrm e(g_{n_k})(y_k)|=0$. 

Since $g(x_k)\to1$ and $g$ is slowly oscillating, $\lim_k\norm{g_{n_k}g_{x_k,m}-g_{x_k,m}}= 0$. By our choice of $m$ and $\varepsilon$, then, eventually, 
\[
\norm{\alpha_{n_k}^\mathrm e(g_{n_k})\alpha_{n_k}^\mathrm e(g_{x_k,m})-\alpha_{n_k}^\mathrm e(g_{x_k,m})}\leq 3\varepsilon.
\]
Since $\varphi(y_k)=x_k$, we have $y_k\in Y_{\alpha^\mathrm e_{n_k},g_{x_k,m},1/4}$, which implies that $|\alpha_{n_k}^\mathrm e(g_{x_k,m})(y_k)|>1/4$, and therefore $|\alpha_{n_k}^\mathrm e(g_{n_k})\alpha_{n_k}^\mathrm e(g_{x_k,m})(y_k)|\geq 1/4-3\varepsilon\geq 1/8$. As $|\alpha_{n_k}^\mathrm e(g_{n_k})(y_k)|\to 0$ when $k\to\infty$ we have a contradiction.
\end{proof}
This shows that the function $\varphi$ induces $\Phi$, which is therefore trivial.
\end{proof}

\section{Nice liftings: Local existence}\label{S.locallifts}

In this section, we show how under $\OCA$ and $\MA$s unital $^*$-homomorphisms between algebras of functions over Higson coronas admit well-behaved local liftings, and in particular that the coarse weak Extension Principle for surjective maps holds. 

Recall that if $X$ is a u.l.f.\ metric space with a distinguished point $x_0$, the sparse sequences $\bar X^\mathrm e$ and $\bar X^\mathrm o$ were defined in Definition~\ref{def:sparsesequences}. If $U\subseteq \nu Y$ is clopen as witnessed by the projection $\chi_U\in C_\nu(Y)$, we let $\tilde U\subseteq Y$ be such that $\pi_Y(\chi_{\tilde U})=\chi_U$. Such $\tilde U$ always exists, as every projection in $\ell_\infty(Y)/c_0(Y)$ lifts to a projection in $\ell_\infty(Y)$, see e.g. \cite[Lemma 3.1.13]{Fa:STCstar}. 

\begin{theorem}\label{thm:local}
Assume $\OCA$ and $\MA$. Let $\Phi\colon C_\nu(X)\to C_\nu(Y)$ be a unital $^*$-homomorphism whose dual is the continuous function $\tilde\varphi\colon \nu Y\to \nu X$. Then there is a clopen set $U\subseteq\nu Y$ such that
\begin{itemize}
\item $\chi_U\Phi\chi_U\colon C_\nu(X)\to C_\nu(\tilde U)$ is of product form on the sparse sequences $\bar X^\mathrm e$ and $\bar X^\mathrm o$, and
\item $\tilde\varphi[\nu Y\setminus U]$ is nowhere dense in $\nu X$.
\end{itemize} 
\end{theorem}

Before moving on to the proof of Theorem~\ref{thm:local}, we isolate few of its corollaries. 
\begin{theorem}\label{thm:OCAwep}
Assume $\OCA$ and $\MA$. Then $\cWEP(s)$ holds.
\end{theorem}
\begin{proof}
Fix u.l.f.\ metric spaces $X$ and $Y$, and a surjective $^*$-homomorphism $\Phi\colon C_\nu(X)\to C_\nu(Y)$ with dual $\tilde\varphi\colon \nu Y\to\nu X$. By Theorem~\ref{thm:local}, we can find a clopen set $U\subseteq\nu Y$ such that $\chi_U\Phi\chi_U\colon C_\nu(X)\to C_\nu(\tilde U)$ is of product form on the sparse sequences $\bar X^\mathrm e$ and $\bar X^\mathrm o$, and $\tilde\varphi[\nu Y\setminus U]$ is nowhere dense in $\nu X$. By Theorem~\ref{thm:localtoglobal}, $\chi_U\Phi\chi_U$ is trivial. This shows that $\cWEP(\Phi)$ holds.
\end{proof}
\begin{proof}[Proof of Theorems~\ref{thm:main} and \ref{thm:main2}]
Fix u.l.f.\ metric spaces $X$ and $Y$, and a surjective $^*$-homomorphism $\Phi\colon C_\nu(X)\to C_\nu(Y)$ with dual $\tilde\varphi\colon \nu Y\to\nu X$. 

To prove Theorem~\ref{thm:main}, assume $\Phi$ is an isomorphism. In this case, the set $U$ given by the $\cWEP$ localised at $\Phi$ equals $\nu Y$, as $\tilde\varphi[\nu Y\setminus U]$ is clopen and nowhere dense, and therefore empty. This shows that $\Phi$ is trivial, and therefore Proposition~\ref{prop:trivialhoms} gives a coarse equivalence $Y\to X$.

For Theorem~\ref{thm:main2}, we simply apply Proposition~\ref{prop:trivialhoms} to the surjective trivial unital $^*$-homomorphism $\chi_U\Phi\chi_U$ and obtain the desired coarse embedding $\tilde U\to Y$.
\end{proof}

A key use of $\OCA$ and $\MA$ in the proof of Theorem~\ref{thm:local} is to get access to two important lifting results for maps between quotient structures (Theorems~\ref{thm:liftben} and \ref{thm:liftcpc}). 
To state them, we first recall some terminology. 
\begin{itemize}
\item If $(M_n,d_n)$ are metric spaces of uniformly bounded diameter, the reduced product $\prod_nM_n/\Fin$ is the metric space obtained by identifying two sequences in $\prod_nM_n$ when they get closer and closer, i.e., quotienting $\prod_n M_n$ by the equivalence relation $\sim$ defined by 
\[
(a_n)\sim (b_n)\text{ if and only if }\lim_n d_n(a_n,b_n)= 0.
\]
 We denote by $\pi_M\colon\prod_nM_n\to\prod_nM_n/\Fin$ the canonical quotient map. We often write $[a]$ for $\pi_M(a)$, for $a\in\prod_nM_n$.
\item Again if $(M_n,d_n)$ are metric spaces of uniformly bounded diameter and $S\in\mathcal P(\mathbb N)/\Fin$, we have a pseudometric $d_S$ on $\prod_nM_n/\Fin$ which generalise to the metric setting the equivalence relation given by `being eventually equal on indices in $S$'. Specifically, if $a=(a_n)$ and $b=(b_n)$ are elements of $\prod_nM_n$, we let 
\[
d_S([a],[b])=\limsup_{n\in \dot S}d_n(a_n,b_n),
\]
where $\dot S$ is any lift of $S$ in $\mathcal P(\bbN)$.
It is routine to check that $d_S$ does not depend on the choice of the representatives for $[a]$ and $[b]$ nor on that of the lift $\dot S$. We write $[a]=_S[b]$ for $d_S([a],[b])=0$.
\item A function $\rho\colon \prod_nM_n/\Fin\to\prod_nN_n/\Fin$ between reduced products of sequences of metric spaces of uniformly bounded diameter is \emph{coordinate fixing} if 
\[
a=_Sb \Rightarrow \rho(a)=_S\rho(b)
\]
for all $a,b\in\prod_nM_n/\Fin$ and $S\in\mathcal P(\bbN)/\Fin$.

We say that $\rho$ is of \emph{product form} if there are maps $\alpha_n\colon M_n\to N_n$ such that $\prod_n\alpha_n\colon\prod_nM_n\to\prod_nN_n$ lifts $\rho$.
\end{itemize}

The following is the main result of \cite{TrivIsoMetric}:
\begin{theorem}\label{thm:liftben}
Assume $\OCA$ and $\MA$. Then all coordinate fixing functions between reduced products of sequences of separable metric spaces of uniformly bounded diameter are of product form.
\end{theorem}
For the next result, we again need additional terminology, and a few reminders. 
\begin{itemize}
\item Recall that $\mathcal P(\bbN)$ is endowed with the natural (Polish) compact metric product topology. All topological terminology refers to such topology. We will in particular focus on nonmeager and dense ideals in $\mathcal P(\bbN)$. For ideals, nonmeagerness can be characterised combinatorially (this is due to Jalali-Naini and Talagrand, see \cite[Theorem~3.10.1]{FarahBook2000}). An ideal $\mathcal I\subseteq\mathcal P(\mathbb N)$ is dense if every infinite $A\subseteq\mathbb N$ contains an infinite set in $\mathcal I$.
\item Let again $M_n$ be a sequence of metric spaces, and fix a point $0_n\in M_n$. If $a=(a_n)\in\prod_nM_n$, the support of $a$ is the set $\supp(a)=\{n\mid a_n\neq 0_n\}$. If $\rho\colon\prod_nM_n/\Fin\to\prod_nN_n/\Fin$ is a function and $\mathcal I\subseteq \mathcal P(\bbN)$ is an ideal, we say that a map $\tilde\rho\colon\prod_nM_n\to\prod_nN_n$ lifts $\rho$ on $\mathcal I$ if
\[
[\tilde\rho(a)]=\rho([a]) \text{ whenever }\supp(a)\in \mathcal I.
\]
\end{itemize}
The following result, although well-known to a few experts, was never specifically stated in these terms. For the reader's convenience, we provide an argument on how it can be derived from the results of \cite{vignati2018rigidity} and \cite{mckenney2018forcing}. If $S\subseteq\bbN$, we write $\chi_S\in\ell_\infty$ for the characteristic function on $S$. Recall that $\mathbb D^\mathbb N$ is identified with the unit ball of $\ell_\infty$.
\begin{theorem}\label{thm:liftcpc}
Assume $\OCA$ and $\MA$. Suppose that $\rho\colon\ell_\infty/c_0\to\ell_\infty/c_0$ is a positive order zero contraction. Then there is a nonmeager dense ideal $\mathcal I\subseteq\mathcal P(\bbN)$ containing all finite sets and a strongly continuous positive contractive order zero map $\tilde\rho\colon\ell_\infty\to\ell_\infty$ lifting $\rho$ on $\mathcal I$.
\end{theorem}
\begin{proof}
Here we use the `noncommutative $\OCA$ lifting theorem' as stated and proved in \cite[Theorem 4.6]{mckenney2018forcing}. Thanks to this result, we can find
\begin{itemize}
\item finite intervals $I_n,J_n\subseteq \bbN$ with the property that $\{I_n\}$ is a partition of $\bbN$ into consecutive finite intervals, and $J_n\cap J_m=\emptyset$ whenever $|n-m|\geq 2$,
\item a nonmeager dense ideal $\mathcal I\subseteq\mathcal P(\bbN)$ containing all finite sets, and
\item a contractive function of product form 
\[
\rho'=\sum\rho'_n\colon\mathbb D^\bbN \to\mathbb D^\bbN
\] 
where for every $k$, if $k\in I_n$ then $\rho'_n\colon \mathbb D\to \mathbb D^{J_n}$, and $\rho'$ lifts $\rho$ on elements supported on $\mathcal I$.
\end{itemize}
Moreover (see \cite[Lemma 3.11]{mckenney2018forcing} or \cite[Proposition 2.18]{vignati2018rigidity}), $\rho'$ satisfies the following property: there is a sequence of positive reals $(\varepsilon_n)$ such that $\lim_n\varepsilon_n= 0$ such that for every $n$ and every $a$ and $b$ in $\mathbb D^\bbN$ supported on $\bigcup_{m\geq n} I_m$ we have that 
\begin{enumerate}
\item $\norm{\rho'(a+b)-\rho'(a)-\rho'(b)}<\varepsilon_n$,
\item for every $\lambda\in\mathbb D$, $\norm{\rho'(\lambda a)-\lambda\rho'(a)}<\varepsilon_n$,
\item if $ab=0$ then $\norm{\rho'(a)\rho'(b)}<\varepsilon_n$, and
\item if $a$ is positive then $\rho'(a)\in\ell_\infty$ only takes values in $[-\varepsilon_n,1]$.
\end{enumerate}
For $k\in\mathbb N$, let $n$ be such that $k\in I_n$. For $\lambda\in\mathbb D$, we let $\tilde\rho_k(\lambda \chi_k)=\lambda(\rho'_n(\chi_n)-2\varepsilon_n)_+$. The map $\tilde\rho=\sum\tilde\rho_n$ is positive, order zero, and equals $\rho'$ modulo $c_0$. This shows that $\tilde\rho$ satisfies the thesis of the theorem.
\end{proof}

Before moving to the proof of Theorem~\ref{thm:local} we record a set theoretic fact. Consider the order of eventual domination $\leq^*$ on $\N^\N$, where $f\leq^* f'$ if and only if there is $n$ such that $f(m)\leq f'(m)$ for all $m\geq n$. The cardinal $\mathfrak b$ is defined to be the smallest cardinality of a $\leq^*$-unbounded subset of $\N^\N$, and $\OCA$ implies that $\mathfrak b=\omega_2$ (\cite[\S8]{Todorcevic.PPIT}). 
\begin{lemma}\label{lem:reverseorder}
If $\mathfrak b>\omega_1$, for every family $\mathcal A\subseteq\N^{\N\uparrow}$ of size $\aleph_1$ we can find $\bar f\in \N^{\N\uparrow}$ such that $\bar f\leq^* f$ for all $f\in\mathcal A$.
\end{lemma}
\begin{proof}
For each $f\in\N^{\N\uparrow}$, let $\tilde f$ be defined by $\tilde f(n)=\min \{k\mid \forall k'\geq k \, ( f(k')\geq n)\}$. If $f\leq^* f'$, then $\tilde f'\leq^*\tilde f$. Since $\mathfrak b>\omega_1$, the family $\tilde{\mathcal A}=\{\tilde f\mid f\in\mathcal A\}$ is $\leq^*$-bounded by some strictly increasing $\bar g$, and we can assume that $\bar g(0)=0$. The function $\bar f$ defined by $\bar f(k)=n$ where $n$ is the maximum natural such that $\bar g(n)\leq k$ is as required.
\end{proof}

We are ready to give the proof of Theorem~\ref{thm:local}. If $g$ and $g'$ are in $\ell_\infty(X)$, we write $g=_{C_0(X)}g'$ for $g-g'\in C_0(X)$. (Similar notation will be used for functions in $\ell_\infty(Y)$).
\begin{proof}[Proof of Theorem~\ref{thm:local}]
Recall, $\OCA$ and $\MA$ are assumed along the way. $\MA$ will only be used to get access to Theorems~\ref{thm:liftben} and \ref{thm:liftcpc}, while $\OCA$ will additionally be used directly towards the end of the proof.

We fix two u.l.f.\ metric spaces $(X,d_X)$ and $(Y,d_Y)$, and a unital $^*$-homomorphism $\Phi\colon C_\nu(X)\to C_\nu(Y)$ whose dual is the continuous function $\tilde\varphi\colon \nu Y\to \nu X$. We also fix a distinguished point $x_0$, and let $\bar X^\mathrm e=(X_n^\mathrm e)$ and $\bar X^\mathrm o=(X_n^\mathrm o)$ be the sparse sequences defined in Definition~\ref{def:sparsesequences}. If $S\subseteq\mathbb N$ and $i\in \{\mathrm e,\mathrm o\}$, we let
\[
X_S^i=\bigcup_{n\in S}X_n^i.
\] 
We need to find $\tilde U\subseteq Y$ such that $\chi_{\tilde U}$ is a slowly oscillating projection (so that this gives a clopen $U\subseteq\nu Y$), $\tilde\varphi[Y\setminus U]$ is nowhere dense, and $\chi_U\Phi\chi_U$ is trivial on the sparse sequences $\bar X^\mathrm e$ and $\bar X^\mathrm o$, meaning that we can find, for $i\in\{\mathrm e,\mathrm o\}$, finite sets $Y_n^{i}\subseteq Y$ and maps $\alpha^{i}_n\colon\mathbb D^{X_n^i}\to\mathbb D^{Y^i_n}$ such that $\Lambda^i=\sum\alpha^i_n$ lifts $\chi_U\Phi\chi_U$ on slowly oscillating functions supported on $\bigcup X^i_n$. 

We first find the sets $Y_n^i$, for $n\in\mathbb N$ and $i\in\{\mathrm e,\mathrm o\}$, and consequently the set $U\subseteq Y$. Secondly, we prove that $U$ has the desired properties and only later we construct the required maps $\alpha_n^i$.

For $i\in\{\mathrm e,\mathrm o\}$, let $g_{n,i}\in\ell_\infty(X)$ be defined by
\[
g_{n,i}(x)=\max \{0,1-\frac{d_X(x,X_n^i)}{n}\}.
\]
Each $g_{n,i}$ has finite support and, for a fixed $i$ and $n\neq m$ we have that $g_{n,i}g_{m,i}=0$. In particular, for every $S\subseteq\mathbb N$, 
\[
g_{S,i}:=\sum_{n\in S}g_{n,i}
\]
is a positive slowly oscillating contraction, hence by Lemma~\ref{lem:orderzero} the map $S\mapsto \pi_X(g_{S,i})$ induces a positive contractive order zero map $\sigma_i\colon\ell_\infty\to\ell_\infty(X)$. Let
\[
\rho_i=\Phi\circ\sigma_i,
\]
so that
\[
\rho_i\colon \ell_\infty/c_0\to C_\nu(Y)\subseteq\ell_\infty(Y)/c_0(Y)
\]
is a positive contractive order zero map. Since we assumed $\OCA$ and $\MA$, we can apply Theorem~\ref{thm:liftcpc} to the map $\rho_i$, and obtain a nonmeager dense ideal $\mathcal I_i$ which contains all finite sets and a strongly continuous positive contractive order zero map
\[
\tilde\rho_i\colon\ell_\infty\to\ell_\infty(Y)
\] 
lifting $\rho_i$ on elements supported on $\mathcal I_i$. For $S\subseteq\bbN$, we let $p_S$ be the canonical projection on $S$ in $\ell_\infty$. We write $p_n$ for $p_{\{n\}}$. Let, for brevity, 
\[
q_{S,i}:=\tilde\rho_i(p_S),
\] 
so that, if $S\in\mathcal I_i$, then $\pi_Y(q_{S,i})=\Phi(\pi_X(g_{S,i}))$. Since $\tilde\rho_i$ is strongly continuous, for $i\in\{\mathrm e,\mathrm o\}$ and $S\subseteq\mathbb N$ we have $q_{S,i}=\sum_{n\in S}q_{n,i}$.

Since $\mathcal I_i$ is dense and for every $S\in \mathcal I_i$ we have that $q_{S,i}$ belongs to $C_\nu(Y)$ (being a lift for $\Phi(\pi_X(g_{S,i}))$), Lemma~\ref{lem:orderzero} (applied to $g=q_{\mathbb N,i}$ and $X_n=\supp(g_{n,i})$) gives that $q_{S,i}$ is slowly oscillating for each $S\subseteq\mathbb N$.
Let
\[
Y_n^i=\{y\in Y\mid q_{n,i}(y)>1/2\}, \, \, \text{ and } \tilde U=\bigcup_n Y_n^\mathrm e\cup Y_n^\mathrm o.
\]
\begin{claim}\label{claim:cutting}
For every $i\in\{\mathrm e,\mathrm o\}$ and $S\subseteq\mathbb N$ we have that $\chi_{\tilde U}q_{S,i}=_{C_0(Y)}q_{S,i}$.
\end{claim}
\begin{proof}
We will prove that $\chi_{\tilde U}q_{\mathbb N,\mathrm e}=_{C_0(Y)}q_{\mathbb N,\mathrm e}$. 
Since $q_{S,\mathrm e}\leq q_{\mathbb N,\mathrm e}$ for all $S\subseteq\bbN$, this suffices. (The proof for $i=\mathrm o$ is exactly the same.)

Suppose thus that $\chi_{\tilde U}q_{\mathbb N,\mathrm e}-q_{\mathbb N,\mathrm e}\notin C_0(Y)$, so that there is $\varepsilon>0$ and a sequence $(x_n)\subseteq Y\setminus\tilde U$ such that $q_{\mathbb N,\mathrm e}(x_n)>\varepsilon$ for all $n$. Since $q_{\mathbb N,\mathrm e}=\sum_n q_{n,\mathrm e}$, all the $q_{n,\mathrm e}$, for $n\in\mathbb N$, are pairwise orthogonal, and the ideal $\mathcal I_\mathrm e$ is dense, by going to a subsequence we can find an infinite $\{k_n\mid n\in\mathbb N\}\in\mathcal I_\mathrm e$ such that $q_{k_n,\mathrm e}(y_n)>\varepsilon$ for all $n$. Let $S=\{k_n\mid n\in\mathbb N\}$. Again by going to a subsequence, we can assume that both $S$ and $T=\{k_n+1\mid n\in\mathbb N\}$ belong to $\mathcal I_\mathrm o\cap\mathcal I_\mathrm e$.

Let $j\colon [0,1]\to[0,1/2]$ be a continuous function such that $j$ is the identity on $[0,1/2]$ and equals $0$ on $[2/3,1]$. Since $g_{S\cup T,\mathrm o}j(g_{S,\mathrm e})=j(g_{S,\mathrm e})$ and $\tilde\rho_\mathrm e$ and $\tilde\rho_\mathrm o$ are liftings for $\rho_\mathrm e$ and $\rho_\mathrm o$ on elements supported on $\mathcal I_\mathrm e\cap\mathcal I_\mathrm o$, we have that
\[
q_{S\cup T,\mathrm o}j(q_{S,\mathrm e})=_{C_0(Y)}j(q_{S,\mathrm e}).
\]
As $p_{S\cup T}=\sum p_{k_n}+p_{k_n+1}$ and $\tilde\rho_\mathrm o$ is order zero and strongly continuous, we have that $q_{S\cup T,\mathrm o}=\sum_nq_{k_n,\mathrm o}+q_{k_n+1,\mathrm o}$. Since 
\[
\liminf_n q_{k_n,\mathrm e}(y_n)=\liminf_n j(q_{k_n,\mathrm e})(y_n)\geq\varepsilon,
\]
we have that
 \[
 \lim_n (q_{k_n,\mathrm o}+q_{k_n+1,\mathrm o})(y_n)=1.
 \]
As $q_{k_n,\mathrm o}q_{k_n+1,\mathrm o}=0$ we have that, eventually, $\max \{q_{k_n,\mathrm o}(y_n),q_{k_n+1,\mathrm o}(y_n)\}\geq 2/3$, which implies that $y_n\in Y_{k_n}^\mathrm o\cup Y_{k_n+1}^\mathrm o\subseteq\tilde U$, a contradiction.
\end{proof}

Note that $q_{\mathbb N,\mathrm e}+q_{\mathbb N,\mathrm o}$, modulo a $C_0(Y)$ function, is a slowly oscillating function which evaluates $\geq1/2$ on $\tilde U$ and vanishes outside of it. This shows that $\chi_{\tilde U}$ is slowly oscillating, and therefore 
\[
U\subseteq\nu Y\text{ is clopen}.
\]
By Claim~\ref{claim:cutting}, substituting $q_{\mathbb N,i}$ with $\chi_{\tilde U}q_{\mathbb N,i}$ we still have that $S\mapsto q_{S,i}$ induces a strongly continuous, positive, and order zero map lifting $\Phi\restriction \{\pi_X(g_{S,i})\mid S\subseteq\mathbb N\}$ on the nonmeager dense ideal $\mathcal I_i$.

\begin{claim}
$\tilde\varphi[\nu Y\setminus U]$ is nowhere dense.
\end{claim}
\begin{proof}
Let $W=\tilde\varphi[\nu Y\setminus U]$. Since $\tilde\varphi$ is continuous and $\nu Y\setminus U$ is clopen, $\tilde\varphi[\nu Y\setminus U]$ is closed. We argue by contradiction and let $V\subseteq W$ be nonempty and open in $\nu X$. Since $V$ is open and included in the image of $\tilde\varphi$, if $f$ is supported on $V$ then $\norm{f}=\norm{\Phi(f)}$. Moreover by the definition of $U$, for all $S\subseteq\mathbb N$ and $i\in\{\mathrm e,\mathrm o\}$, if $f$ is a contraction supported on $V$ then $\norm{\Phi(f)\pi_Y(q_{S,i})}\leq 1/2$.

 Let $f\in C_\nu(X)$ be a norm one positive contraction supported on $V$, so that $\norm{\Phi(f)}=1$. Let $\tilde f$ be a positive contraction with $\pi_X(\tilde f)=f$, and pick an infinite sequence $(x_n)\subseteq X$ such that $\tilde f(x_n)\to 1$ as $n\to\infty$. By passing to a subsequence if necessary, we can assume that there is a sequence of naturals $(k_n)$ and $i\in\{\mathrm e,\mathrm o\}$ such that $x_n\in\mathrm{int}(X_{k_n}^i,n)$ and $S=\{k_n\mid n\in\mathbb N\}\in\mathcal I_i$. Multiplying $\tilde f$ by a positive contraction supported on $\bigcup X_{k_n}^i$ we can assume that $\tilde f=\sum \tilde f_n$ where $\tilde f_n$ is positive and supported on $X_{k_n}^i$. Note that $\tilde f=\tilde fg_{S,i}$. Since $S\in\mathcal I_i$ we have that $\pi_Y(q_{S,i})=\Phi(\pi_X(g_{S,i}))$, hence
 \[
1=\norm{\Phi(f)}=\norm{\Phi(f)\Phi(\pi_X(g_{S,i}))}=\norm{\Phi(f)\pi_Y(q_{S,i})}\leq 1/2.
\]
This contradiction concludes the proof.
\end{proof}
We are left to show that the $^*$-homomorphism 
\[
\Phi^U:=\chi_U\Phi\chi_U\colon C_\nu(X)\to C_\nu(\tilde U)
\]
 is of product form on the sparse sequences $\bar X^\mathrm e$ and $\bar X^\mathrm o$. We fix a set theoretic lifting $\tilde \Phi^U\colon C_h(X)\to \chi_{\tilde U}C_h(Y)\chi_{\tilde U}$ for $\Phi^U$, and also fix $i\in\{\mathrm e,\mathrm o\}$.
%

If $g\in \ell_\infty(X)$ is supported on $\bigcup_nX_n^i$ and $S\subseteq\mathbb N$, we write $g_S$ for $g\chi_{X_S}=g\restriction X_S^i$. By Lemma~\ref{lem:orderzero}, the function $g$ is slowly oscillating if and only if each $g_S$ is. 
 
\begin{claim}\label{claim:rightsupport}
For every $i\in\{\mathrm e,\mathrm o\}$, $g\in C_h(X)$ supported on $\bigcup X_n^i$, and $S\subseteq \mathbb N$ we have that
\[
\chi_{Y_S^i}\tilde \Phi^U(g)=_{C_0(Y)}\tilde\Phi^U(g_S).
\]
\end{claim}
\begin{proof}
Fix $g$, $i$, and $S$. We will show that 
\[
\chi_{Y_S^i}\tilde \Phi^U(g_{\bbN\setminus S})=_{C_0(Y)}0=_{C_0(Y)} \chi_{Y_S^i} \tilde \Phi^U(g_{S})-\tilde \Phi^U(g_{S}),
\]
 so that, since $\tilde\Phi^U(g)=_{C_0(Y)}\tilde\Phi^U(g_S)+\tilde\Phi^U(g_{\mathbb N\setminus S})$, we have that
\[
\chi_{Y_S^i}\tilde \Phi^U(g)-\tilde\Phi^U(g_S)=_{C_0(Y)}\chi_{Y_S^i}(\tilde\Phi^U(g_S)+\tilde\Phi^U(g_{\mathbb N\setminus S}))-\tilde\Phi^U(g_S)=_{C_0(Y)}0.
\]
Suppose that $\chi_{Y_S^i}\tilde \Phi^U(g_{\bbN\setminus S})\notin C_0(Y)$. Let $(y_n)\subseteq Y_S^i$ be an infinite sequence such that $|\tilde\Phi^U(g_{\mathbb N\setminus S})(y_n)|>\varepsilon$ for some $\varepsilon>0$ and all $n\in\mathbb N$. By passing to a subsequence we can suppose there is an infinite $\{k_n\mid n\in\mathbb N\}\subseteq S$ belonging to $\mathcal I_i$ such that $y_n\in Y_{k_n}^i$. Let $T=\{k_n\mid n\in\mathbb N\}$. By the definition of $Y_{k_n}^i$, we have that $q_{T,i}(y_n)\geq 1/2$. Since $y_n$ is an infinite sequence and $T\in\mathcal I_i$, altogether we get
\[
0=\norm{\Phi^U(\pi_Y(q_{T_i}g_{\mathbb N\setminus S}))}\geq\liminf |q_{T,i}\tilde \Phi^U(g_{\mathbb N\setminus S})(y_n)|\geq\varepsilon/2,
\]
a contradiction. 

We now show that $\chi_{Y_S^i} \tilde \Phi^U(g_{S})-\tilde \Phi^U(g_{S})\in C_0(Y)$. Assume $i=\mathrm e$. Suppose by contradiction that there is $\varepsilon>0$ and a sequence $(y_n)\subseteq U\setminus Y_S^\mathrm e$ such that $|\tilde \Phi^U(g_{S})(y_n)|>\varepsilon$. Since $\chi_{Y_{\mathbb N\setminus S}^\mathrm e}\tilde \Phi^U(g_{S})\in C_0(Y)$ we have that $(y_n)\subseteq U\setminus Y_\mathbb N^\mathrm e\subseteq Y_\mathbb N^\mathrm o$. Let $(k_n)$ be an increasing sequence of naturals such that $y_n\in Y_{k_n}^\mathrm o$. Let 
\[
T=\{k_n\mid n\in\mathbb N\}\text{ and }T^+=\{k_n+1\mid n\in\mathbb N\}.
\]
 By going to a subsequence, we can assume that both $T$ and $T^+$ belong to $\mathcal I_\mathrm o\cap\mathcal I_\mathrm e$. Now, notice that $g_{T,\mathrm o}g_{S\setminus (T\cup T^+)}=0$, and therefore $q_{T,\mathrm o}\tilde\Phi^U(g_{S\setminus (T\cup T^+)})\in C_0(Y)$, which implies that $\lim|\tilde\Phi^U(g_{S\setminus (T\cup T^+)})(y_n)|=0$, and therefore 
\[
\liminf |\tilde\Phi^U(g_{T\cup T^+})(y_n)|\geq\varepsilon.
\]
Since $q_{T\cup T^+,\mathrm e}$ lifts $\Phi^U(\pi_X(g_{T\cup T^+,\mathrm e}))$ and $g_{T\cup T^+,\mathrm e}g_{T\cup T^+}=g_{T\cup T^+}$, we have that $\liminf q_{T\cup T^+,\mathrm e}(y_n)=1$. As all the $q_{\ell,\mathrm e}$ are pairwise orthogonal, for every large enough $n$ there is $\ell_n$ such that $q_{\ell_n,\mathrm e}(y_n)\geq 2/3$, i.e., $y_n\in Y_\mathbb N^\mathrm e$. This is a contradiction. The proof for $i=\mathrm o$ is exactly the same.
\end{proof}

We modify $\tilde \Phi^U$ in the following way: if $g\in C_h(X)$ is supported on $\bigcup X_n^i$ for some $i\in\{\mathrm e,\mathrm o\}$, we let $S_g=\{n\mid g\restriction X_n^i\neq 0\}$, and we redefine $\tilde\Phi^U(g)$ by $\chi_{Y_{S_g}^i}\tilde\Phi^U(g)$.  By Claim~\ref{claim:rightsupport}, this redefined $\tilde\Phi^U$ is still a lift for $\Phi^U$. This function is coordinate preserving in the following sense: if $g$ and $g'$ are both supported on $\bigcup X_{n}^i$ and $g_S=_{C_0(X)}g'_S$ for some $S\subseteq\mathbb N$ then $\tilde\Phi^U(g)\restriction Y_S^i=_{C_0(Y)}\tilde\Phi^U(g')\restriction Y_S^i$. This shows that if $g\in C_h(X)$ is supported on $\bigcup X_n^i$, then 
\[
\Phi^U(\pi_X(g))\in \prod_n \mathbb D^{Y_n^i}/\bigoplus \mathbb D^{Y_n^i}
\]
and that, for every $S\in \mathcal P(\mathbb N)/\Fin$ and $g,g'$ supported on $\bigcup X_n^i$,
\[
\text{ if } \pi_X(g)_S=\pi_X(g')_S\text{ then } \Phi^U(\pi_X(g))_S= \Phi^U(\pi_X(g'))_S.
\]
We are tempted to use Theorem~\ref{thm:liftben} at this moment. The only problem is that, if one looks at the set of slowly oscillating functions supported on $\bigcup X_n^i$ for a fixed $i\in\{\mathrm e,\mathrm o\}$, this is not a reduced product. For this reason, we return to the sets of the form $M_f$, for $f\in\mathbb N^{\mathbb N\uparrow}$, given after Definition~\ref{def:productformspseq}. Fix $i\in\{\mathrm e,\mathrm o\}$ and let
\[
M_{n,m}=M_{n,m,X_n^i}, \,\,\, M_{f}=M_{f,\bar X^i} (\text{for } f\in \N^{\N\uparrow}) \text{, and }N_{f}=M_{f}/\bigoplus M_{n,f(n)}.
\]
Note that $\{g\in C_h(X)\mid \supp(g)\subseteq\bigcup X_n^i\}=\bigcup_{f\in\N^{\N\uparrow}}M_{f}$. We will thus apply Theorem~\ref{thm:liftben} to 
\[
\Phi_f=\Phi^U\restriction N_f\colon N_f\to \prod\mathbb D^{Y_n^i}/\bigoplus \mathbb D^{Y_n^i}.
\]
and then uniformize the liftings using an $\OCA$ argument.

Fix $f\in\N^{\N\uparrow}$. By Claim~\ref{claim:rightsupport}, $\Phi_f$ is coordinate fixing, meaning that the hypotheses of Theorem~\ref{thm:liftben} are satisfied. $\OCA$, $\MA$, and Theorem~\ref{thm:liftben} give us functions $\alpha_{f,n}\colon M_{n,f(n)}\to \mathbb D^{Y^i_n}$ such that 
\[
\prod_n\alpha_{f,n}\colon M_{f}\to \prod_n \mathbb D^{Y_n^i}
\]
lifts $\Phi_f$. 
\begin{claim}
For every $f,f'\in\N^{\N\uparrow}$ and $\varepsilon>0$ there is $n=n(f,f',\varepsilon)$ such that for all $m\geq n$ and $g\in M_{f,m}\cap M_{f',m}$ we have that $\norm{\alpha_{f,m}(g)-\alpha_{f',m}(g)}<\varepsilon$.
\end{claim}
\begin{proof}
Suppose not, and fix an infinite sequence $(n_k)$ such that for every $k$ there is $g_k\in M_{n_k,f(n_k)}\cap M_{n_k,f'(n_k)}$ such that $\norm{\alpha_{f,n_k}(g_k)-\alpha_{f',n_k}(g_k)}>\varepsilon$. Let $g=\sum g_k$, and note that $g\in M_f\cap M_{f'}$. Since $(n_k)$ is an infinite sequence, we have that $\prod_n\alpha_{f,n}(g)-\prod_n\alpha_{f',n}(g)\notin C_0(Y)$. This contradicts that both $\prod_n\alpha_{f,n}(g)$ and $\prod_n\alpha_{f',n}(g)$ lift $\Phi(\pi_X(g))$.
\end{proof}

For a fixed $\bar X$ and $\varepsilon>0$, colour $[\N^{\N\uparrow}]^2=K_0^\varepsilon\sqcup K_1^\varepsilon$ by $\{f,f'\}\in K_0^{\varepsilon}$ if and only if
\[
\exists n \exists g\in M_{n,f(n)}\cap M_{n,f'(n)} (\norm{\alpha_{f,n}(g)-\alpha_{f',n}(g)}>\varepsilon).
\]
Let $\mathrm{Maps}(\mathbb D^{X_n^i},\mathbb D^{Y_n^i})$ be the set of all maps from $\mathbb D^{X_n^i}$ to $\mathbb D^{Y_n^i}$. This is a Polish space with the $\sup$-distance, and so we can endow $\prod_n\mathrm{Maps}(\mathbb D^{X_n^i},\mathbb D^{Y_n^i})$ a Polish topology. Looking at $K_0^\varepsilon$ as a subset of $[\prod_n\mathrm{Maps}(\mathbb D^{X_n^i},\mathbb D^{Y_n^i})\times \N^{\N\uparrow}]^2$ shows that $K_0^\varepsilon$ is an open colouring, and therefore we can apply $\OCA$ to it. (This should be compared with \cite[Lemma 6.7]{mckenney2018forcing} or \cite[Theorem 17.8.2]{Fa:STCstar}.)

\begin{claim}\label{claim:uniformising}
For every $\varepsilon>0$ there is no uncountable $K_0^\varepsilon$-homogeneous set.
\end{claim}
\begin{proof}
Suppose that $\mathcal A$ is $K_0^\varepsilon$-homogeneous and of size $\aleph_1$. Since $\mathfrak b>\omega_1$ we can find $\bar f$ such that $\bar f\leq^* f$ for all $f\in\mathcal A$ (Lemma~\ref{lem:reverseorder}). By refining $\mathcal A$, we can assume there is $n_0$ such that $\bar f(n)\leq f(n)$ whenever $n\geq n_0$ and $f\in\mathcal A$. A further refinement of $\mathcal A$ allows us to suppose that the natural $\bar n=n(\bar f,f,\varepsilon/2)$ is the same for all $f\in \mathcal A$. We enlarge it so that $\bar n\geq n_0$. Since $\prod_{n\leq\bar n}\{f\colon \mathbb D^{X_n^i}\to\mathbb D^{Y_n^i}\}$ is separable (endowed with the product topology), we can assume that for all $f,f'\in\mathcal A$ and $n\leq\bar n$ we have that $\norm{\alpha_{f,n}-\alpha_{f',n}}<\varepsilon$. This is our final refinement. 

Pick now $f,f'\in\mathcal A$, and find $n$ and $g\in M_{n,f(n)}\cap M_{n,f'(n)}$ witnessing that $\{f,f'\}\in K_0^\varepsilon$. By our choice of $\mathcal A$, $n>\bar n\geq n_0$, and therefore $\bar f(n)\leq f(n),f'(n)$, meaning that $g\in M_{n,\bar f(n)}$. By our choice of $\bar n$, we have that 
\begin{align*}
\norm{\alpha_{f,n}(g)-\alpha_{f',n}(g)}&\leq\norm{\alpha_{f,n}(g)-\alpha_{\bar f,n}(g)}+\norm{\alpha_{\bar f,n}(g)-\alpha_{f',n}(g)}\\&< \varepsilon/2+\varepsilon/2=\varepsilon.
\end{align*}
This shows that $\{f,f'\}\notin K_0^\varepsilon$, contradicting the homogeneity of $\mathcal A$.
\end{proof}
By applying $\OCA$ to the open colouring $K_0^\varepsilon\subseteq [\bbN^\bbN]^2$, we can find sets $\mathcal Y_{n,k}$ such that each $\mathcal Y_{n,k}$ is $K_1^{2^{-k}}$-homogeneous and $\N^{\N\uparrow}=\bigcup_n\mathcal Y_{n,k}$. 

The following uniformization proof is similar (in spirit) to the one of \cite[Theoremn 17.8.2]{Fa:STCstar} (see also \cite{Fa:All}).
\begin{claim}\label{claim1}
There are sets $\mathcal Y_k\subseteq\N^{\N\uparrow}$ with $\mathcal Y_k\supseteq\mathcal Y_{k+1}$ and an increasing sequence of naturals $(n_k)$ such that each $\mathcal Y_k$ is $K_1^{2^{-k}}$-homogeneous and
\[
\forall k \forall f\in\N^{\N\uparrow}\exists f'\in \mathcal Y_k (\forall n\text{ if } f(n)\geq n_k \text{ then } f'(n)<f(n)).
\]
\end{claim}
\begin{proof}
Consider the poset $(\mathbb N^{\mathbb N\uparrow},R)$, where $fRg$ if and only if $g\leq^* f$. This is a $\sigma$-directed poset (i.e., every countable subset of it has an upper bound). Furthermore, $R=\bigcup R_n$ where we define $R_n$ by $fR_n g$ if and only if for all $k$, if $f(k)\geq m$ then $g(k)\leq f(k)$. A recursive argument (see e.g. the proof of Lemma 6.8 in \cite{mckenney2018forcing}), using that if a $\sigma$-directed poset is partitioned into countably many pieces then one of such pieces must be cofinal (\cite[Lemma 2.2.2]{FarahBook2000}), allows us to construct an increasing sequence of naturals $(n_k)$ and sets $\mathcal Y_k$ such that each $\mathcal Y_k$ is $K_1^{2^{-k}}$-homogeneous and $\mathcal Y_k$ is $R_{n_k}$-cofinal. This is the thesis.
\end{proof}
Fix a sequence of naturals $(n_k)$ as provided by Claim~\ref{claim1}, and let $n_{-1}=0$. We are ready to define a map $\Lambda=\prod\alpha_n\colon \prod_n\mathbb D^{X_n^i}\to \prod_n\mathbb D^{Y_n^i}$ which lifts $\Phi^U$ on contractions in $C_h(X)$ supported on $\bigcup X_n^i$.

 Fix $n\in\mathbb N$ and a contraction $g\in C_h(X)$ supported on $X_n^i$. If $g$ is constant, just send $g$ to the constant function (with the same value as $g$ has) on $Y_n$. If not, since $g\in M_{n,0}$, $M_{n,m}\subseteq M_{n,m'}$ whenever $m'\leq m$, and $\bigcap_m M_{n,m}\subseteq\{g\mid g\text{ is constant}\}$, there is a maximal $m=m_g$ such that $g\in M_{n,m}\setminus M_{n,m+1}$. Let $k$ be such that $n_k\leq m_g<n_{k+1}$, and find $f=f_{g,n}\in\mathcal Y_k$ such that $f(n)\leq m_g$, so that $g\in M_{n,f(n)}$. Let 
 \[
 \alpha_n(g)=\alpha_{f,n}(g).
 \]
\begin{claim}
$\Lambda=\sum\alpha_n$ lifts $\Phi^U$ on $\{g\in C_h(X)\mid\supp(g)\subseteq \bigcup X_n^i\}$.
\end{claim}
\begin{proof}
Fix a slowly oscillating contraction $g$ supported on $\bigcup X_n^i$, and write $g=\sum g_n$ where $\supp(g_n)\subseteq X_n^i$ for all $n$. Fix $k>0$, and let $f$ be such that $g\in M_f$ and $f\in \mathcal Y_k$, so that $\sum_n\alpha_{f,n}(g_n)$ lifts $\Phi(g)$. We will show that
\[
\norm{\pi_Y(\sum_n\alpha_{f,n}(g_n)-\Lambda(g))}=\limsup\norm{\alpha_{f,n}(g_n)-\alpha_n(g_n)}<2^{-k}.
\]
Since $f\in\N^{\N\uparrow}$, we can find $n$ such that $f(m)\geq n_k$ for all $m\geq n$. For every $m\geq n$, we have that the natural $k_{g_m}$ is greater than or equal to $k$, and therefore $f_{g_m,m}\in\mathcal Y_{k_{g_m}}\subseteq\mathcal Y_k$, which gives that
\[
\norm{\alpha_m(g_m)-\alpha_{f,m}(g_m)}=\norm{\alpha_{f_{g_m,m},m}(g_m)-\alpha_{f,m}(g_m)}<2^{-k}.
\]
Since this happens for every $m\geq n$, we have that 
\[
\limsup\norm{\alpha_{f,n}(g_n)-\alpha_n(g_n)}\leq 2^{-k}.
\]
As $k$ is arbitrary, this gives the thesis.
\end{proof}
The last claim shows that $\Phi^U$ is trivial on the sparse sequence $\bar X^i$. Since $i$ is arbitrary, this concludes the proof.
\end{proof}

\bibliographystyle{amsplain}
\bibliography{bibliography}
\end{document}